\documentclass[11pt]{amsart}
\theoremstyle{plain}
\newtheorem{thm}{Theorem}[section]
\newtheorem{theorem}[thm]{Theorem}

\newtheorem{lemma}[thm]{Lemma}
\newtheorem{corollary}[thm]{Corollary}

\theoremstyle{definition}
\newtheorem{remark}[thm]{Remark}

\numberwithin{equation}{section}

 \title[Relative dynamical degree of correspondences over an arbitrary field]{Relative dynamical degrees of correspondences over a field of arbitrary characteristic}
 \author{Tuyen Trung Truong}
 \address{School of Mathematical Sciences, The University of Adelaide, Adelaide SA 5005, Australia}
 \email{tuyen.truong@adelaide.edu.au}
\thanks{The author was supported by Australian Research Council grants DP120104110 and DP150103442.}
    \date{\today}

\begin{document}
\maketitle
\begin{abstract}
Let $K$ be an algebraically closed field of arbitrary characteristic, $X$ an irreducible variety and $Y$ an irreducible projective variety over $K$, both are not necessarily smooth. Let $f:X\rightarrow X$ and $g:Y\rightarrow Y$ be dominant correspondences, and $\pi :X\rightarrow Y$ a dominant rational map such that $\pi \circ f=g\circ \pi$.  We define relative dynamical degrees $\lambda _p(f|\pi )$ ($p=0,\ldots ,\dim (X)-\dim (Y)$). These degrees measure the relative growth of positive algebraic cycles, satisfy a product formula when $Y$ is smooth and $g$ is a multiple of a rational map, and are birational invariants. More generally, a weaker product formula is proven for more general semi-conjugacies, and for any  generically finite semi-conjugacy $(\varphi ,\psi )$ from $(X_2,f_2)\rightarrow (Y_2,g_2)$ to $(X_1,f_1)\rightarrow (Y_1,g_1)$ we have $\lambda _p(f_1|\pi _1)\geq \lambda _p(f_2|\pi _2)$ for all $p$. Many of our results are new even when $K=\mathbb{C}$. We make use of  de Jong's alterations and Roberts' version of Chow's moving lemma. In the lack of resolution of singularities, the consideration of correspondences is necessary even when $f,g$ are rational maps. The case $K$ is not algebraically closed further requires working with correspondences over reducible varieties. 

\end{abstract}

\section{Introduction}
 
One important tool in Complex Dynamics is dynamical degrees for dominant meromorphic selfmaps. They are bimeromorphic invariants of a meromorphic selfmap $f:X\rightarrow X$ of a compact K\"ahler manifold $X$. The $p$-th dynamical degree $\lambda _p(f)$  is the exponential growth rate of the spectral radii of the pullbacks $(f^n)^*$ on the Dolbeault  cohomology group $H^{p,p}(X)$. For a surjective holomorphic map $f$, the dynamical degree $\lambda _p(f)$ is simply the spectral radius of $f^*:H^{p,p}(X)\rightarrow H^{p,p}(X)$.  Fundamental results of Gromov \cite{gromov} and Yomdin \cite{yomdin} expressed the topological entropy of a surjective holomorphic map in terms of its dynamical degrees: $h_{top}(f)=\log \max _{0\leq p\leq {\dim (X)}}\lambda _{p}(f)$. Since then, dynamical degrees have played a more and  more important role in dynamics of meromorphic maps. In many results and conjectures in Complex Dynamics in higher dimensions, dynamical degrees play a central role. 

Let $X$ be a compact K\"ahler manifold of dimension $k$ with a K\"ahler form $\omega _X$, and let $f:X\rightarrow X$ be a dominant meromorphic map. For $0\leq p\leq k$, the $p$-th dynamical degree $\lambda _p(f)$ of $f$ is defined as follows 
\begin{equation} 
\lambda _p(f)=\lim _{n\rightarrow\infty}(\int _X(f^n)^*(\omega _X^p)\wedge \omega _X^{k-p})^{1/n}=\lim _{n\rightarrow\infty}r_p(f^n)^{1/n},
\label{Equation01}\end{equation}
where $r_p(f^n)$ is the spectral radius of the linear map $(f^n)^*:H^{p,p}(X)\rightarrow H^{p,p}(X)$. The existence of the limit in the above expression is non-trivial and has been proven by Russakovskii and Shiffman \cite{russakovskii-shiffman} when $X=\mathbb{P}^k$, and by Dinh and Sibony \cite{dinh-sibony10}\cite{dinh-sibony1} when $X$ is compact K\"ahler. Both of these results use regularisation of positive closed currents. The limit in (\ref{Equation01}) is important in showing that dynamical degrees are birational invariants. 

For meromorphic maps of compact K\"ahler manifolds with invariant fibrations, a more general notion called relative dynamical degrees has been defined by Dinh and Nguyen in \cite{dinh-nguyen}. (Here, by a fibration we simply mean a dominant rational map, without any additional requirements.) Via "product formulas" proven in \cite{dinh-nguyen} when $X$ is complex projective and \cite{dinh-nguyen-truong1} when $X$ is compact K\"ahler, these relative dynamical degrees provide a very useful tool to check whether a meromorphic map is primitive (i.e.  has no invariant fibrations over a base which is of smaller dimension and not a point).  Roughly speaking, primitive maps (first defined in \cite{zhang}) are those which do not come from smaller dimensional manifolds, hence are "building blocks" from which all meromorphic maps can be constructed. In another direction,  when $K=\mathbb{C}$,  Dinh and Sibony \cite{dinh-sibony11} defined dynamical degrees and topological entropy for meromorphic correspondences over irreducible varieties. Computations of dynamical degrees of so-called Hurwitz correspondences of the moduli spaces $\mathcal{M}_{0,N}$  were given in \cite{ramadas}, wherein a proof that dynamical degrees of correspondences (over $K=\mathbb{C}$, and for irreducible varieties) are birational invariants was also given. 

Recently, work on birational maps of surfaces over an algebraically closed field of arbitrary characteristic has become more and more popular. As some examples, we refer the readers  to \cite{esnault-srinivas, xie, blanc-cantat, esnault-oguiso-yu, oguiso}. In these results, (relative) dynamical degrees also play an important role. Because of this and  for further applications in algebraic dynamics, it is desirable to have a purely algebraic definition of (relative) dynamical degrees for a rational map over a field other than $\mathbb{C}$.  The purpose of this paper is to define (relative) dynamical degrees for dominant rational maps over arbitrary fields and varieties. It turns out that the same idea can be applied to {\bf dominant correspondences} as well.  

{\bf Notation.} Before continuing, let us recall some notation. Let $K$ be a field and $X,Y$ irreducible (not necessarily smooth or projective) varieties. A {correspondence}  $f:X\rightarrow Y$ is given by an algebraic cycle $\Gamma _f=\sum _{i=1}^m\Gamma _i$ on $X\times Y$, where $m$ is a positive integer and $\Gamma _i\subset X\times Y$ are irreducible subvarieties of dimension exactly $\dim (X)$. We do not assume that $\Gamma _i$ are distinct, and hence may write the above sum as $\sum _{j}a_j\Gamma _j$ where $\Gamma _j$ are distinct and $a_j$ are positive integers. We will call $\Gamma _f$ the graph of $f$, by abusing the usual notation when $f$ is a rational map.  If $f$ is a correspondence and $a\in \mathbb{N}$, we denote by $af$ the correspondence whose graph is $a\Gamma _f$. In other words, if $\Gamma _f=\sum _{i}\Gamma _i$ then $\Gamma _{af}=\sum _{i}a\Gamma _i$.  If $\Gamma _f=a\Gamma$ where $\Gamma$ is irreducible and $a\in \mathbb{N}$, we say that the correspondence $f$ is {irreducible}. A  rational map $f$ is an irreducible correspondence, since its graph is irreducible. A correspondence is dominant if for each $i$ in the sum, the two natural projections from $\Gamma _i$ to $X,Y$ are dominant. Dominant correspondences can be composed and the resulting correspondence is also dominant. Given two dominant correspondences $f:X\rightarrow X$ and $g:Y\rightarrow Y$, we say that they are semi-conjugate if there is a dominant rational {\bf map} $\pi :X\rightarrow Y$ such that $\pi \circ f=g\circ \pi$. We will simply write $\pi :(X,f)\rightarrow (Y,g)$ to mean that $\pi$ is a dominant rational map semi-conjugating $(X,f)$ and $(Y,g)$. 

When the field $K$ is such that resolutions of singularities exist for varieties over $K$ up to dimension $d$, we established previously in \cite{truong} the existence of relative dynamical degrees of semi-conjugate dominant rational maps over spaces of dimension $\leq d$ and some basic properties. Let us remind briefly the ideas of the proof in this case. Given two dominant rational maps $f:X\rightarrow X$ and $g:Y\rightarrow Y$ which are semi-conjugated via a dominant rational map $\pi :X\rightarrow Y$, we can by using resolutions of singularities to reduce to the case where $X$ and $Y$ are smooth and $\pi$ is regular. The results are established provided certain quantitative estimates of strict transforms by rational maps and strict intersections of varieties are available. To establish the needed estimates, in \cite{truong} we used  a version of the Chow's moving lemma by Roberts \cite{roberts}. By Hironaka's theorem, resolutions of singularities exist in all dimensions when $K$ has characteristic $0$. On a field of positive characteristic, resolutions of singularities are established for dimensions up to $3$ (\cite{zariski1, zariski2, abhyankar1, abhyankar2, cutkosky, cossart-piltant1, cossart-piltant2}) through the effort of many mathematicians. It is still an open question whether resolution of singularities exists for varieties in higher dimensions. This restricts the scope of  applications of results in \cite{truong}.

Our current paper overcomes this difficulty in positive characteristics by using - in addition to Roberts' version of Chow's moving lemma as in our previous work - de Jong's alteration and (relative) dynamical degrees of correspondences. Let us give a brief description of this idea. By seminal results of de Jong \cite{jong}, an alteration is always available. More precisely, given an irreducible variety $X$, there always exist an irreducible smooth quasi-projective variety $X_1$ (having a smooth projective compactification) together with a generically finite, surjective regular morphism $\varphi :X_1\rightarrow X$.  Now, let $f:X\rightarrow X$ be a dominant rational map. To define dynamical degrees for $f$, in order to apply the strategy described above, the first step would be to find a dominant rational map $f_1:X_1\rightarrow X_1$ so that we have a semi-conjugacy $\varphi : (X_1,f_1)\rightarrow (X,f)$. On the one hand, the existence of such a map $f_1$ is not guaranteed when $f$ and $\varphi$ are fixed. On the other hand, a correspondence $f_1$ can readily be found so that $\varphi :(X_1,f_1)\rightarrow (X,\deg (\varphi )f)$ is a semi-conjugacy. If we can define dynamical degrees for the correspondence $f_1$,  we can use them to define dynamical degrees for $f$. To prove that these dynamical degrees are well-defined, we need to show that they do not depend on the choice of the alteration $X_1$. This means that, in the lack of resolution of singularities, we need also prove that dynamical degrees are invariant for certain {\bf generically finite} semi-conjugate correspondences. A special case of the needed result has been proven (\cite{dinh-nguyen}) for dynamical degrees of {generically finite} semi-conjugate meromorphic maps of compact K\"ahler manifolds.    

Even if the varieties we are concerned are all projective and {smooth}, in the lack of resolution of singularities, the consideration of correspondences is still necessary if we want to define relative dynamical degrees for a semi-conjugacy of rational maps $\pi :(X,f)\rightarrow (Y,g)$. In fact, to define relative dynamical degrees, the first step is to reduce to the case $\pi$ is regular. We can apply de Jong's alteration to the graph of $\pi$ to have that $\pi$ is regular, but in exchange we have to consider correspondences. As far as we know, relative dynamical degrees for correspondences have not been proposed in the current literature, even in the case  $K=\mathbb{C}$. When $K$ is not algebraically closed, we have to further work with correspondences over reducible varieties, even if the starting varieties are irreducible.    
 
Now we are ready to state the main results of the paper. For simplicity, we will state here only the case when $K$ is algebraically closed and $X,Y$ are projective varieties. Extensions (e.g. $X$ is not projective or $K$ is not algebraically closed) will be discussed later.   
\begin{theorem}
Let $X$ and $Y$ be irreducible projective (not necessarily smooth) varieties, of corresponding dimensions $k$ and $l$, over an algebraic closed field $K$. Let $\pi :(X,f)\rightarrow (Y,g)$ be a semi-conjugacy of  dominant correspondences. 

Then,  for any $0\leq p\leq k-l$, there is a well-defined birational invariant number $\lambda _p(f|\pi )$, called the $p$-th relative dynamical degree. If $Y=$ a point, we denote these simply by $\lambda _p(f)$ ($p=0,\ldots ,k$) and call them dynamical degrees. These relative dynamical degrees have the following properties: 

1) {\bf The regular case.} Assume that $X$ is smooth and $\pi $ is regular. Let $\iota _X:X\rightarrow \mathbb{P}^N$ and $\iota _Y :Y\rightarrow \mathbb{P}^M$ be embeddings. Let $H_M\subset \mathbb{P}^M$ and $H_N\subset \mathbb{P}^N$ be ample divisors, and define $\omega _X=\iota _X^*(H_N)$. Then for each $p=0,\ldots ,k-l$ the limit
\begin{eqnarray*}
\lim _{n\rightarrow\infty}\deg ((f^n)^*(\omega _X^*(H_N^p)).(\iota _Y \circ \pi)^*(H_N^l))^{1/n},
\end{eqnarray*}
exists and is independent of the choices of $\iota _X$, $\iota _Y$ and $H_M,H_N$ . Here the degree of an algebraic cycle on $X$ is computed  in the embedding $\iota _X$. We define $\lambda _p(f|\pi )$ to be this limit.  We always have $\lambda _p(f|\pi )\geq 1$ for all $p=0,\ldots ,k-l$.  

2) {\bf Birational invariant.} The relative dynamical degrees are birational invariants. More precisely, let be given two semi-conjugacies $\pi _1:(X_1,f_1)\rightarrow (Y,g)$ and $\pi _2:(X_2,f_2)\rightarrow (Y_2,g_2)$, together with two  semi-conjugacies $\varphi :(X_2,f_2)\rightarrow (X_1,f_1)$ and $\psi : (Y_2,g_2)\rightarrow (Y_1,g_1)$ where the two composition maps $\psi \circ \pi _2$ and $\pi _1\circ \varphi$ are the same. Assume that $\psi :Y_2\rightarrow Y_1$ is a generically finite rational map, and moreover that $\varphi$ is birational.  Then $\lambda _p(f_1|\pi _1)=\lambda _p(f_2|\pi _2)$ for all $p=0,\ldots ,\dim (X_1)-\dim (Y_1)$.

 3) {\bf Log-concavity.}  Assume that there are infinitely many $n\in \mathbb{N}$  such that the correspondences $f^n$ are irreducible, i.e. $f^n=d_n\Gamma _n$ where $\Gamma _n$ is an irreducible varitety and $d_n\in \mathbb{N}$. Then, the relative dynamical degrees are log-concave, that is $\lambda _{p}(f|\pi )^2\geq \lambda _{p-1}(f|\pi ).\lambda _{p+1}(f|\pi )$. In particular, if $f$ is a dominant rational map, then its dynamical degrees are log-concave. 
 
 4) {\bf Product formula.} Assume that $Y$ is smooth, and $g=ag'$ where $a$ is a positive integer and $g':Y\rightarrow Y$ a dominant rational map. Then, for every $0\leq p\leq k$ we have
\begin{eqnarray*}
a\lambda _p(f) = \max _{0\leq j\leq l,~0\leq p-j\leq k-l}\lambda _j(g)\lambda _{p-j}(f|\pi ),
\end{eqnarray*}
or equivalently
\begin{eqnarray*}
\lambda _p(f)= \max _{0\leq j\leq l,~0\leq p-j\leq k-l}\lambda _j(g')\lambda _{p-j}(f|\pi ).
\end{eqnarray*}

In particular, if $f$ is a dominant rational map (hence, so is $g$), then the product formula is satisfied with $a=1$. 
 
5) {\bf Computing via numerical equivalent classes.} Assume that $X$ is smooth. Let $N^p(X)$ be the group of algebraic cycles of codimension $p$ on $X$ modulo numerical equivalence (see the next section for more details), and $N^p_{\mathbb{R}}(X)=N^p(X)\otimes _{\mathbb{Z}}\mathbb{R}$. Let $||.||$ be any norm on $N^p_{\mathbb{R}}(X)$, and $||(f^n)^*_p||$ the corresponding norm of the linear map $(f^n)^*:N^p_{\mathbb{R}}(X)\rightarrow N^p_{\mathbb{R}}(X)$. Then
\begin{eqnarray*}
\lambda _p(f)=\lim _{n\rightarrow\infty}||(f^n)^*_p||^{1/n}.
\end{eqnarray*}
\label{TheoremRelativeDynamicalDegrees}\end{theorem}
{\bf Remark.} There are abundant of examples for which the theorem can be applied. There are also reducible correspondences $f$, for which the log-concavity fails for its dynamical degrees.  See Remark \ref{Remark1} for more detail. Note that if $\pi :(X,f)\rightarrow (Y,g)$ is a semi-conjugacy, and $f$ is a multiple of a rational map, then so is $g$. However, the reverse is false, that is there are semi-conjugacies where $g$ is  a multiple of a rational map but $f$ is not,  see Remark \ref{Remark1}. 

Under the assumptions of part 4) of Theorem \ref{TheoremRelativeDynamicalDegrees}, it can be easily checked that $\lambda _0(g)=a$, and hence the product formula can be written as follows
\begin{eqnarray*}
\lambda _0(g)\lambda _p(f) = \max _{0\leq j\leq l,~0\leq p-j\leq k-l}\lambda _j(g)\lambda _{p-j}(f|\pi ),
\end{eqnarray*}
for all $p=0,\ldots ,k$. A generalisation of the product formula, stated in terms of inequalities instead of equalities, holds for more general correspondences. 
\begin{theorem}
Let $\pi :(X,f)\rightarrow (Y,g)$ be a semi-conjugacy between dominant correspondences. Assume that $Y$ is smooth and there is an increasing sequence $n_i\in \mathbb{N}$ such that $g^{n_i}$ are irreducible. Then,
\begin{eqnarray*}
\lambda _0(g)\lambda _p(f) \geq \max _{0\leq j\leq l,~0\leq p-j\leq k-l}\lambda _j(g)\lambda _{p-j}(f|\pi ),
\end{eqnarray*}
for all $p=0,\ldots ,\dim (X)$. 

Moreover, if $c$ is any positive number so that 
\begin{eqnarray*}
c\lambda _p(f) \geq \max _{0\leq j\leq l,~0\leq p-j\leq k-l}\lambda _j(g)\lambda _{p-j}(f|\pi ),
\end{eqnarray*}
for all $p=0,\ldots ,\dim (X)$, then $c\geq \lambda _0(g)$.
\label{TheoremWeakerProductFormula}\end{theorem}
{\bf Remark.} From the proof of Theorem \ref{TheoremWeakerProductFormula}, it seems that the theorem does not hold for general semi-conjugacies (over irreducible varieties). At the moment, we cannot  yet find any counter-example for correspondences over irreducible varieties. In the mean time, a counter-example is given in Remark \ref{Remark1} for a totally disjoint union of semi-conjugacies over reducible varieties.  

The next result gives a stronger version of part 2) in Theorem \ref{TheoremRelativeDynamicalDegrees} in which the assumption that $\pi $ is birational is lessen to  requiring only that it is generically finite. It says that we always have $\lambda _p(f_1|\pi _1)\geq \lambda _p(f_2|\pi _2)$, and under some additional conditions equality happens. 
\begin{theorem}
Let be given two semi-conjugacies $\pi _1:(X_1,f_1)\rightarrow (Y_1,g_1)$ and $\pi _2:(X_2,f_2)\rightarrow (Y_2,g_2)$, together with two  semi-conjugacies $\varphi :(X_2,f_2)\rightarrow (X_1,f_1)$ and $\psi : (Y_2,g_2)\rightarrow (Y_1,g_1)$ where the two composition maps $\psi \circ \pi _2$ and $\pi _1\circ \varphi$ are the same. Assume that $\varphi :X_2\rightarrow X_1$ and $\psi :Y_2\rightarrow Y_1$ are generically finite rational maps. 

1) Then we have $\lambda _p(f_1|\pi _1)\geq \lambda _p(f_2|\pi _2)$ for all $p=0,\ldots ,\dim (X_1)-\dim (Y_1)$.

2) Assume moreover that at least one of the following conditions  is satisfied:

a) $X_1$ is smooth and $\pi _1$ is regular;

or 

b) Resolution of singularities is available for varieties of dimension up to $\dim (X_1)$. 

Then $\lambda _p(f_2|\pi _2)=\lambda _p(f_1|\pi _1)$ for all $p=0,\ldots ,\dim (X_1)-\dim (Y_1)$.
\label{TheoremFurtherProperty}\end{theorem}

Relative dynamical degrees also satisfy a type of triangle inequality. This can be used to produce examples of primitive correspondences, discussed in more details in Section 6.      
\begin{theorem}
Let $f_1,f_2:X\rightarrow X$ be two commutative correspondences (that is, $f_1\circ f_2=f_2\circ f_1$), and $\pi :X\rightarrow Y$ a dominant rational map. Then $\lambda _p(f_1+f_2|\pi )\leq \lambda _p(f_1|\pi )+\lambda _p(f_2|\pi )$ for all $p$.
\label{TheoremTriangleInequality}\end{theorem}
We remark that the inequality is strict in general. See Remark \ref{Remark1} for more detail. 

{\bf Main idea for the proof of Theorem \ref{TheoremRelativeDynamicalDegrees}.}  

Given any correspondence $f:X\rightarrow X$ and a generically finite rational map $\varphi :Z\rightarrow X$, we can construct a pullback correspondence $\varphi ^*f:Z\rightarrow Z$ so that $\varphi :(Z,\varphi ^*(f))\rightarrow (X,\deg (\varphi )f)$ (see Section 3 for more details). 

By using some appropriate alterations $\tau _1:X_1\rightarrow X$, we can find a semi-conjugacy $\pi _1:(X_1,\tau _1^*f)\rightarrow (Y,\deg (\tau _1)g)$, and moreover the map $\pi _1:X_1\rightarrow Y$ is regular. Using part 1) of Theorem \ref{TheoremRelativeDynamicalDegrees}, we then would like to define
\begin{eqnarray*}
\lambda _p(f|\pi )=\frac{1}{\deg (\tau _1)}\lambda _p(\tau _1^*f|\pi _1). 
\end{eqnarray*}
For this to be well-defined, we need to show that this quantity does not depend on the choice of the alterations. To this end, given any two pullback semi-conjugacies $\tau _1:(X_1,\tau _1^*f)\rightarrow (X,\deg (\tau _1)f)$ and $\tau _2:(X_2,\tau _2^*f)\rightarrow (X,\deg (\tau _2)f)$, we first construct an irreducible smooth projective variety $X_3$ together with generically finite regular morphisms $\varphi _1:X_3\rightarrow X_1$, $\varphi _2:X_3\rightarrow X_2$ and $\tau _3:X_3\rightarrow X$ such that $\varphi _1\circ \tau _1=\tau _3=\varphi _2\circ \tau _2$. (For example this can be done by choosing $X_3$ to be an alteration of a component of the pullback of the diagonal $\Delta _X\subset X\times X$ by the generically finite regular morphism $\tau _1\times \tau _2:X_1\times X_2\rightarrow X\times X$.)  It follows from this that $\deg (\varphi _1)\deg (\tau _1)=\deg (\tau _3)=\deg (\varphi _2)\deg (\tau _2)$. Then by results in Section 3 of this paper, we have two semi-conjugacies $\varphi _1: (X_3,\tau _3 ^*f)\rightarrow (X_1,\deg (\varphi _1)\tau _1^*f)$ and $\varphi _2: (X_3,\tau _3^*f)\rightarrow (X_2,\deg (\varphi _2)\tau _2^*f)$. These semi-conjugacies allow us to relate both $\lambda _p(\tau _1^*f|\pi _1)$ and $\lambda _p(\tau _2^*f|\pi _2)$ to $\lambda _p(\tau _3^*f|\pi _3)$, via the help of part 1). 

This proves the existence of relative dynamical degrees. Part 1) is proven in the course of this process. Part 2) is proven by observing that if $\varphi :(X_2,f_2)\rightarrow (X_1,f_1)$ is a semi-conjugacy, where $\varphi$ is birational, then $f_2=\varphi ^*f_1$. For part 3), let $n_i$ be a sequence of increasing integers so that $f^{n_i}$ are irreducible. If $X$ is smooth, then we can apply the Grothendieck-Hodge index theorem to an alteration of $\Gamma _{f^{n_i}}$, together with part 1) of Theorem \ref{TheoremRelativeDynamicalDegrees} to complete the proof for part 3). When $X$ is not smooth, we observe that if $\tau _1:X_1\rightarrow X$ is an alteration,  then while $(\tau _1^*(f))^{n_i}$ may not be irreducible, the pullbacks by different irreducible components of it are related in such a way that the Grothendieck-Hodge index theorem again applies. Part 3) is optimal, in the sense that there are examples of reducible correspondences whose dynamical degrees do not satisfy the log-concavity, see Remark \ref{Remark1} below.  Part 4), as in \cite{dinh-nguyen, dinh-nguyen-truong1}, essentially follows from estimates on strict intersection and strict pullback by correspondences of algebraic cycles which will be given in Section 4. Note that there are complications when we consider general correspondences $f$ and $g$, since the equality $f^*(\alpha \wedge \beta )=f^*(\alpha )\stackrel{o}{\wedge}f^*(\beta )$ for appropriate positive cycles, which were used in \cite{dinh-nguyen, dinh-nguyen-truong1} for rational maps $f$, simply do not hold for a general correspondence. We are able to prove a weaker version of this necessary intermediate result under the assumption that $g$ is a multiple of a rational map. Part 5) is standard.

\begin{remark} 
\label{Remark1}  
We now give some further remarks on Theorem \ref{TheoremRelativeDynamicalDegrees}.

1) In the case of complex projective manifolds, the dynamical degrees can be also computed using Dolbeault cohomology groups $H^{p,p}(X)$. The proof of Theorem \ref{TheoremRelativeDynamicalDegrees} in that setting is given using properties of positive closed currents. For a smooth projective variety defined on an arbitrary algebraically closed field of characteristic zero, the analogue of $H^{p,p}(X)$ is the algebraic de Rham groups $H^p(X,\Omega _X^p)$. Since we lack the notions of positivity for the classes in $H^p(X,\Omega _X^p)$, it is not clear that dynamical degrees can also be computed using algebraic de Rham groups. The same question can be asked for other cohomology groups, for example $l$-adic cohomology groups. We note that for {\bf automorphisms of surfaces} over an algebraic closed field of arbitrary characteristic, dynamical degrees can be computed by using the $l$-adic cohomology groups as proven in Esnault and Srinivas \cite{esnault-srinivas}.   
    
2) In the literature, the relative dynamical degrees for rational maps were given over fields where resolution of singularities exists (\cite{dinh-nguyen, dinh-nguyen-truong1, truong}) and the dynamical degrees for correspondences were given when $K=\mathbb{C}$ (\cite{dinh-sibony11}). With some efforts, results on fields of characteristic $0$ may be reduced to the case where $K=\mathbb{C}$ via the Lefschetz principle. The pullback of an effective divisor is effective, hence the first dynamical degree $\lambda _1(f)$ of a dominant rational map $f:X\rightarrow X$ (see e.g. \cite{boucksom-favre-jonsson} where the case of surfaces was discussed) and by duality also the dynamical degree $\lambda _{k-1}(f)$ (here $k=\dim (X)$) are well-defined, over arbitrary characteristic. 

3) The following features of our results are new: 

- Correspondences over reducible varieties and their dynamical degrees, even when $K=\mathbb{C}$.

- Relative dynamical degrees for correspondences,  even when $K=\mathbb{C}$ and the involved varieties are irreducible.

- Dynamical degrees for rational maps over a field of arbitrary characteristic.

 - A completely different proof of geometric flavour - for the known results over $K=\mathbb{C}$ - whose estimates  are given explicitly in terms of the dimensions of $X,Y$ and the degrees of them in  given embeddings $X,Y\subset \mathbb{P}^N$.
 
 - Theorem \ref{TheoremFurtherProperty} is stated for the first time, even for the case where $K=\mathbb{C}$ (in this case \cite{dinh-nguyen} observed it for the special case of dynamical degrees of meromorphic maps of compact K\"ahler manifolds). 
 
 - Theorems \ref{TheoremWeakerProductFormula} and \ref{TheoremTriangleInequality} are also stated for the first time, even when $K=\mathbb{C}$. 
  
4) There are abundant of examples for which the above results can be applied. Let $\pi :X\rightarrow Y$ be a dominant rational map and $g:Y\rightarrow Y$ a dominant correspondence. Let $k=\dim (X)$, $l=\dim (Y)$, $\varphi :X\rightarrow \mathbb{P}^{k-l}$ a dominant rational map and $h:\mathbb{P}^{k-l}\rightarrow \mathbb{P}^{k-l}$ a dominant correspondence. Let $\tau :X\rightarrow \mathbb{P}^{k-l}$ be a dominant rational map, we can pullback - via the generically finite dominant rational map $(\pi ,\tau ):X\rightarrow Y\times \mathbb{P}^{k-l}$ (see Section 3 for more details) - the correspondence $g\times h:Y\times \mathbb{P}^{k-l}\rightarrow Y\times\mathbb{P}^{k-l}$ to a correspondence $f:X\rightarrow X$, together with a semi-conjugacy $\pi :(X,f)\rightarrow (Y,a.\deg (\pi ,\tau )g)$, where $a$ is the cardinality of $h(z)$ for a generic point $z\in \mathbb{P}^{k-l}$. 
 
 5) The log-concavity is proven as soon as a version of Grothendieck-Hodge index theorem is available for the graphs of an infinite subset of $\{f^n:~n\in \mathbb{N}\}$, thanks to the existence of the limit in Part 1) of Theorem \ref{TheoremRelativeDynamicalDegrees}. Provided that $f^n$ is irreducible, we can use Grothendieck-Hodge index theorem for an alteration of the graph of $f^n$. This is where the condition of the irreducibility of $f^n$ comes in. The reason for the log-concavity to fail in the case $f$ is reducible lies in the fact that the Grothendieck-Hodge index theorem fails for reducible smooth varieties. 

It is certainly true that the assumption of part 3 in Theorem \ref{TheoremRelativeDynamicalDegrees} does not hold for a general correspondence. For example \cite{dinh-sibony11}, if $f$ is an irreducible correspondence which is symmetric with respect to the diagonal, then $f^2$ contains a multiple of the diagonal, hence is not irreducible. However, it is expected that the assumption holds in many cases of interest, other than the case where $f$ is already a rational map. In fact, we expect that for a generic irreducible correspondence $f:X\rightarrow X$, all iterates $f^n$ are irreducible.

The log-concavity actually fails for a general correspondence even when $X$ is smooth. Here is one counter-example in the case where $f$ is not irreducible. Let $g:X\rightarrow X$ be a dominant rational map such that there is a $p=1,\ldots ,\dim (X)-1$ for which $\lambda _{p-1}(g)+\lambda _{p+1}(g)>2 \lambda _p(g)$. (For example, $g$ is an endomorphism of degree $\geq 2$ on $\mathbb{P}^k$.) Let $a\geq 1$ be an integer, and $f=g+a\Delta _X$, where $\Delta _X$ is the diagonal.  It can be computed that $\lambda _q(f)=\lambda _q(g)+a$ for all $q=0,\ldots ,\dim (X)$. Then, from the assumptions on the dynamical degrees of $g$, it follows that for this $f$ we have $\lambda _{p}(f)^2<\lambda _{p-1}(f)\lambda _{p+1}(f)$ provided $a$ is large enough.  

6) Here we discuss the product formula (part 4) of Theorem \ref{TheoremRelativeDynamicalDegrees} and Theorem \ref{TheoremWeakerProductFormula}) in case the correspondences $g^n$ are always reducible. The proof of Theorem \ref{TheoremWeakerProductFormula}) seems to indicate that its result is false in this case, even if $X$ and $Y$ are irreducible varieties. While we cannot yet find a counter-example in this case, it is easy to construct one in the case the varieties $X$ and $Y$ are reducible. In fact, let $\pi _1:(X_1,f_1)\rightarrow (Y_1,g_1)$ and $\pi _2:(X_2,f_2)\rightarrow (Y_2,g_2)$ be semi-conjugacies of correspondences for which Theorem \ref{TheoremWeakerProductFormula} holds, where $X_1,X_2$ are distinct irreducible varieties of dimension $k$ and $Y_1,Y_2 $ are distinct irreducible varieties of dimension $l$. Let $X=X_1\cup X_2$, $Y=Y_1\cup Y_2$ and $\pi :X\rightarrow Y$ is given by the pair $\pi _1,\pi _2$. Let $f:X\rightarrow X$ be the correspondence given by the pair $f_1,f_2$ and $g:Y\rightarrow Y$ the correspondence given by the pair $g_1,g_2$.  Then it can be checked that $\pi \circ f^n=g^n\circ \pi$ for all $n\in \mathbb{N}$, by the definition in Section 6. Moreover, we have
\begin{eqnarray*}
\lambda _p(f)&=&\max \{\lambda _p(f_1),\lambda _p(f_2)\},\\
\lambda _p(f|\pi )&=&\max\{\lambda _p(f_1|\pi _1),\lambda _p(f_2|\pi _2)\},\\
\lambda _p(g)&=&\max\{\lambda _p(g_1),\lambda _p(g_2)\}. 
\end{eqnarray*}

Let $X_1=\mathbb{P}^3$ and $X_2=$ blowup of $\mathbb{P}^3$ at one point, $Y_1=\mathbb{P}^2$ and $Y_2=$ blowup of $\mathbb{P}^2$ at one point. Let $\pi _1:X_1\rightarrow Y_1$ be the map $[x_0:x_1:x_2:x_3]\mapsto [x_0:x_1:x_2]$, and $\pi _2:X_2\rightarrow Y_2$ the lift of $\pi _1$. Let $f_1:X_1\rightarrow X_1$ be the reverse of, i.e. its graph is the image under the automorphism $(z_1,z_2)\mapsto (z_2,z_1)$ on $X_1\times X_1$ of that of, the map $[x_0:x_1:x_2:x_3]\mapsto [x_0^{d_1}:x_1^{d_1}:x_2^{d_1}:x_3^{d_1}]$, and $g_1$ is the reverse of the map $[x_0:x_1:x_2]\mapsto [x_0^{d_1}:x_1^{d_1}:x_2^{d_1}]$. It can be checked that $\pi _1\circ f_1^n=g_1^n\circ \pi _1$ for all $n\in \mathbb{N}$. Moreover, 
\begin{eqnarray*}
\lambda _0(f_1)&=&d_1^3,~\lambda _1(f_1)=d_1^2,\\
\lambda _0(g_1)&=&d_1^2,~\lambda _1(g_1)=d_1.
\end{eqnarray*}  
 
Similarly, by working with appropriate dominant rational maps on $\mathbb{P}^2$ which have $\lambda _1=\lambda _2=d_2$, we can construct correspondences $f_2:X_2\rightarrow X_2$ and $g_2:Y_2\rightarrow Y_2$ so that $\pi _2\circ f_2^n=g_2^n\circ \pi _2$ for all $n\in\mathbb{N}$, and moreover
\begin{eqnarray*}
\lambda _0(f_2)&=&d_2,~\lambda _1(f_2)=d_2,\\
\lambda _0(g_2)&=&d_2,~\lambda _1(g_2)=d_2.
\end{eqnarray*}  

If Theorem \ref{TheoremWeakerProductFormula} is to hold for $\pi :(X,f)\rightarrow (Y,g)$, we would have in particular
\begin{eqnarray*}
\max \{\lambda _0(g_1),\lambda _0(g_2)\}\max \{\lambda _1(f_1),\lambda _1(f_2)\}&\geq& \max \{\lambda _1(g_1),\lambda _1(g_2)\}\max \{\lambda _0(f_1|\pi _1),\lambda _0(f_2|\pi _2)\}\\
&=&\max \{\lambda _1(g_1),\lambda _1(g_2)\}\max \{\lambda _0(f_1),\lambda _0(f_2)\}.
\end{eqnarray*}
(The equality in the above expression follows from the proof of Theorem \ref{TheoremPrimitiveCorrespondence}.) We will however show that this is not the case. In fact, assume that it were. Then, we would have
\begin{eqnarray*}
\max \{d_1^2,d_2\}\max \{d_1^2,d_2\}\geq \max \{d_1,d_2\}\max \{d_1^3,d_2\}.
\end{eqnarray*}

We note that in the above constructions, $d_1,d_2$ can be chosen to be any positive integer. If we choose them so that $d_1^3>d_2>d_1^2$, we see - as desired - that the above inequality cannot hold. 

7) If $f:X\rightarrow X$ is any correspondence and $\Delta _X$ is the diagonal, then it can be checked that $\lambda _p(f+\Delta _X)=\lambda _p(f)+1$ for all $p=0,\ldots ,\dim (X)$. Hence the inequality in Theorem \ref{TheoremTriangleInequality} may be the equality in certain examples.  

However, the inequality is strict in general. For example, let $f:X\rightarrow X$ be a birational map so that there exists one number $p$ for which $\lambda _p(f)+\lambda _p(f^{-1})>2$. (There are many such maps, e.g. among pseudo-automorphisms in dimension $3$. The latter are birational maps $f:X\rightarrow X$, where $\dim (X)=3$, so that both $f$ and $f^{-1}$ have no exceptional divisors.) Then $f$ and $f^{-1}$ commute,  and 
\begin{eqnarray*}
(f+f^{-1})^n&=&\sum _{i=0}^nC(n,i)f^if^{-n+i}\\
&\leq&C(n,[n/2])\Delta _X+ \sum _{i<n/2}C(n,i)(f^{-1})^{n-2i}+\sum _{i>n/2}C(n,i)f^{2i-n},
\end{eqnarray*} 
where $C(n,i)$'s are binomial coefficients and $\Delta _X$ the diagonal. By using part 1) of Theorem \ref{TheoremRelativeDynamicalDegrees}, it readily follows that $$\lambda _p(f+f^{-1} )\leq \max \{2,\lambda _p(f )+\frac{1}{\lambda _p(f )},\lambda _p(f^{-1} )+\frac{1}{\lambda _p(f^{-1})}\}<\lambda _p(f )+\lambda _p(f^{-1}).$$

8) The assumption that $K$ is algebraically closed and $X,Y$ are projective is essential in several parts in the proofs of our results. First, to be able to define strict transform of correspondences, we need to use that generic fibers of surjective regular morphisms are of correct dimensions (This replaces the Sard's theorem on singularities of regular maps, which holds only in characteristic zero.), see Section 3 for more details. This property is known only for algebraically closed fields and for quasi-projective varieties. Second, we need to use algebraic cycles and the groups $N^p(X),~N^p(Y)$ which have nice properties under the mentioned assumptions on $K$ and $X,Y$. 

However, assuming $Y$ be projective, we can reduce the general case to the case where $K$ is algebraically closed and $X$ is projective. We will elaborate more on this later in the paper. 
\end{remark}

{\bf Organisation of the paper.} The remaining of this paper is organised as follows. In Section 2, we present some preliminaries on algebraic cycles, Roberts' version of Chow's moving lemma and Grothendieck-Hodge index theorem.  In Section 3 we present some facts about correspondences, semi-conjugacies between them and pullbacks of certain diagrams of dominant correspondences. In Section 4, we give some estimates on strict pullbacks of algebraic cycles by correspondences and strict intersection of algebraic cycles. Section 5 is devoted to the proofs of main results. In Section 6, we will discuss the definition of (relative) dynamical degrees when the varieties involved may be reducible or not projective or when the field $K$ is not algebraically closed, and give some applications to primitive correspondences and the simplicity of the first dynamical degree of a rational map. In particular, we will show that all correspondences of surfaces whose dynamical degrees violate the log-concavity (as constructed in Remark \ref{Remark1}) are weakly primitive, and discuss a possible application of our results to the Hurwitz correspondences considered in \cite{ramadas}. We remark that a weakly primitive correspondence (as defined in Section 6) may still be imprimitive if primitivity is defined in the strongest sense, that is to have no semi-conjugacy to any correspondence $g:Y\rightarrow Y$ where $0<\dim (Y)<\dim (X)$. See Example 2 in Section 6 for more detail. 

{\bf Remark.} After an earlier version of this paper was sent out, we were informed by Charles Favre that Nguyen-Bac Dang has been developing an alternative approach for relative dynamical degrees of rational maps on normal projective  varieties. 

{\bf Acknowledgements.} We would like to thank Finnur Larusson for suggesting using de Jong's alteration, which is very crucial in the treatment of this paper. We thank Keiji Oguiso for checking thoroughly several earlier versions of this paper, his interest in the results of the paper and constant encouragement in the course of this work. We would also like to thank Tien-Cuong Dinh, H\'el\`ene Esnault, Charles Favre, Mattias Jonsson, Pierre Milman  and Claire Voisin for their invaluable help.

\section{Algebraic cycles}

Throughout the section, we fix an algebraic closed field $K$ of {\bf arbitrary} characteristic. Recall that a projective manifold over $K$ is a non-singular subvariety of a projective space $\mathbb{P}_K^N$. We will recall the definition and some results on algebraic cycles, the Roberts' version of Chow's moving lemma and the Grothendieck-Hodge index theorem. We then arrive at a useful result on the intersection of two cycles and define a norm $||.||_1$ which will be used in the proof of Theorem  \ref{TheoremRelativeDynamicalDegrees}.

\subsection{Algebraic cycles}
Let $X\subset \mathbb{P}_K^N$ be a projective manifold of dimension $k$ over an algebraic closed field $K$ of arbitrary characteristic. A $q$-cycle on $X$ is a finite sum $\sum n_i[V_i]$, where $V_i$ are $q$-dimensional irreducible subvarieties of $X$ and $n_i$ are integers. The group of $q$-cycles on $X$, denoted $Z_q(X)$, is the free abelian group on the $p$-dimensional subvarieties of $X$ (see Section 1.3 in Fulton \cite{fulton}). A $q$-cycle $\alpha$ is effective if it has the form
\begin{eqnarray*}
\alpha =\sum _i n_i[V_i],
\end{eqnarray*}
where $V_i$ are irreducible subvarieties of $X$ and $n_i\geq 0$.

Let $X$ and $Y$ be projective manifolds, and let $f:X\rightarrow Y$ be a morphism. For any irreducible subvariety $V$ of $X$, we define the pushforward $f_*[V]$ as follows. Let $W=f(V)$. If $\dim (W)<\dim (V)$, then $f_*[V]=0$. Otherwise, $f_*[V]=\deg (V/W)[W]$. This gives a pushforward map $f_*:Z_q(X)\rightarrow Z_q(Y)$ (see Section 1.4 in \cite{fulton}). 

We refer the readers to \cite{fulton} for the definitions of rational and algebraic equivalences of algebraic cycles. Roughly speaking, two algebraic cycles are rationally equivalent if they are elements of a family of algebraic cycles parametrized by $\mathbb{P}^1$. Similarly, two algebraic cycles are algebraically equivalent if they are elements of  a family of algebraic cycles parametrized by a smooth algebraic variety. The groups of  $q$-cycles modulo rational and algebraic equivalences are denoted by $A_q(X)$ and $B_q(X)$.  

We write $Z^p(X)$, $A^p(X)$ and $B^p(X)$ for the corresponding groups of cycles of codimension $p$. Since $X$ is smooth, we have an intersection product $A^p(X)\times A^q(X)\rightarrow A^{p+q}(X)$, making $A^*(X)$ a ring, called the Chow's ring of $X$ (see Sections 8.1 and 8.3 in \cite{fulton}).  

For a dimension $0$ cycle $\gamma =\sum _im_i[p_i]$ on $X$, we define its degree to be $\deg (\gamma )=\sum _im_i$. We say that a cycle $\alpha \in A^{p}(X)$ is numerically equivalent to zero if and only if $\deg (\alpha .\beta )=0$ for all $\beta \in A^{k-p}(X)$ (see Section 19.1 in \cite{fulton}). The group of codimension $p$ algebraic cycles modulo numerical equivalence is denoted by $N^p(X)$. These are finitely generated free abelian groups (see Example 19.1.4 in \cite{fulton}). The first group $N^1(X)$ is a quotient of the Neron-Severi group $NS(X)=B^1(X)$. The latter is also finitely generated, as proven by Severi and Neron. We will use the vector spaces $N^p_{\mathbb{R}}(X)=N^p(X)\otimes _{\mathbb{Z}}\mathbb{R}$ and $N^p_{\mathbb{C}}(X)=N^p(X)\otimes _{\mathbb{Z}}\mathbb{C}$ in defining dynamical degrees and in proving analogs of  Theorems 1.1 and 1.2 in \cite{truong2}.

\begin{remark}
We have the following inclusions: rational equivalence $\subset $ algebraic equivalence $\subset$ numerical equivalence.
\end{remark}

\subsection{Roberts' version of Chow's moving lemma}
\label{SectionRoberts}

Let $X$ be a projective manifold of dimension $k$ over $K$. If $V$ and $W$ are two irreducible subvarieties of $X$, then either $V\cap W=\emptyset$ or any irreducible component of $V\cap W$ has dimension at least $\dim (V)+\dim (W)-k$. We say that $V$ and $W$ are properly intersected if any component of $V\cap W$ has dimension exactly $\dim (V)+\dim  (W)-k$. When $V$ and $W$ intersect properly, the intersection $V.W$ is well-defined as an effective $\dim (V)+\dim (W)-k$ cycle. 

Given $\alpha =\sum _{i}m_i[V_i]\in Z_q(X)$ and $\beta =\sum _{j}n_j[W_j]\in Z_{q'}(X)$, we say that $\alpha .\beta$ is well-defined if every component of $V_i\cap W_j$ has the correct dimension. The usual statement of Chow's moving lemma is that we can always find $\alpha '$ which is rationally equivalent to $\alpha$ so that $\alpha '.\beta $ is well-defined. Since in the sequel we will need to use some quantitative properties of such cycles $\alpha '$,  we will follow the construction of such cycles $\alpha '$ by Roberts \cite{roberts}. See also the paper \cite{friedlander-lawson} for a generalisation to moving families of cycles of bounded degrees. 

Fixed an embedding $X\subset \mathbb{P}^N_K$, we choose a linear subspace $L\subset \mathbb{P}^N_K$ of dimension $N-k-1$ such that $L\cap X=\emptyset$. For any irreducible subvariety $Z$ of $X$, it follows that $L\cap Z=\emptyset$ as well and we denote by $C_L(Z)$ the cone over $Z$ with vertex $L$ (see Example 6.17 in the  book Harris \cite{harris}). For any such $Z$, $C_L(Z).X$ is well-defined and has the same dimension as $Z$, and moreover $C_L(Z).X-Z$ is effective (see Lemma 2 in \cite{roberts}).

Let $Y_1,Y_2,\ldots , Y_m$ and $Z$ be irreducible subvarieties of $X$, here $Y_1,\ldots ,Y_m$ do not need to be closed. We define the excess $e(Z)$ of $Z$ relative to $Y_1,\ldots ,Y_m$ to be the maximum of the integers $$\dim (W)+k-\dim (Z)-\dim (Y_i),$$ 
where $i$ runs from $1$ to $m$, and $W$ runs through all components of $Z\cap Y_i$, provided that one of these integers is non-negative. Otherwise, the excess is defined to be $0$. 

More generally, if $Z=\sum _im_i[Z_i]$ is a cycle, where $Z_i$ are irreducible subvarieties of $X$, we define $e(Z)=\max _ie(Z_i)$. We then also define the cone $C_L(Z)=\sum _im_iC_{L}(Z_i)$. 
         
The Main Lemma (page $93$) in \cite{roberts} says that for any cycle $Z$ and any irreducible subvarieties $Y_1,\ldots ,Y_m$, then $e(C_L(Z).X-Z)\leq \max (e(Z)-1,0)$ for generic linear subspace $L\subset \mathbb{P}^N$ of dimension $N-k-1$ such that $L\cap X=\emptyset$.

Now we can finish the proof of Chow's moving lemma as follows (see Theorem page 94 in \cite{roberts}). Given $Y_1,\ldots ,Y_m$ and $Z$ irreducible varieties on $X$. If $e=e(Z)=0$ then $Z$ intersect properly $Y_1,\ldots ,Y_m$, hence we are done. Otherwise, $e\geq 1$. Applying the main lemma, we can find linear subspaces $L_1,\ldots ,L_e\subset \mathbb{P}^N_K$ of dimension $N-k-1$, such that if $Z_0=Z$ and $Z_i=C_{L_i}(Z_{i-1}).X-Z_{i-1}$ for $i=1,\ldots ,e=e(Z)$, then $e(Z_i)\leq e-i$. In particular, $e(Z_e)=0$. It is easy to see that
\begin{eqnarray*}
Z=Z_0=(-1)^eZ_e+\sum _{i=1}^{e}(-1)^{i-1}C_{L_i}(Z_{i-1}).X.    
\end{eqnarray*} 

It is known that there are points $g\in Aut(\mathbb{P}_K^N)$ such that $(gC_{L_i}(Z_{i-1})).X$ and $(gC_{L_i}(Z_{i-1})).Y_j$ are well-defined for $i=1,\ldots ,e$ and $j=1,\ldots ,m$. We can choose a rational curve in $Aut(\mathbb{P}_K^N)$ joining the identity map $1$ and $g$, thus see that $Z$ is rationally equivalent to
\begin{eqnarray*}
Z'=(-1)^eZ_e+\sum _{i=1}^{e}(-1)^{i-1}(gC_{L_i}(Z_{i-1})).X. 
\end{eqnarray*}
By construction, $e(Z')=0$, as desired. 

\subsection{Grothendieck-Hodge index theorem}

Let $X\subset \mathbb{P}^N_K$ be a projective manifold of dimension $k$. Let $H\subset \mathbb{P}^N_K$ be a hyperplane, and let $\omega _X =H|_X$. We recall that $N^p(X)$, the group of codimension $p$ cycles modulo the numerical equivalence, is a finitely generated free abelian group. We define $N^p_{\mathbb{R}}(X)=N^p(X)\otimes _{\mathbb{Z}}\mathbb{R}$ and $N^p_{\mathbb{C}}(X)=N^p(X)\otimes _{\mathbb{Z}}\mathbb{C}$. These are real (and complex) vector spaces of real (and complex) dimension equal $rank(N^p(X))$. For $p=1$, it is known that $\dim  _{\mathbb{R}} (N^1_{\mathbb{R}}(X)) =rank(NS(X))=:\rho$, the rank of the Neron-Severi group of $X$ (see Example 19.3.1 in \cite{fulton}). 

We define for $u,v\in N^1_{\mathbb{C}}(X)$ the Hermitian form 
\begin{eqnarray*}
\mathcal{H}(u,v)=\deg (u.\overline{v}.\omega _X^{k-2}).
\end{eqnarray*}
Here the degree of a complex $0$-cycle $\alpha +i\beta$ is defined to be the complex number $\deg (\alpha )+i\deg (\beta )$. The analogue of Hodge index theorem for complex  projective manifolds is the Grothendieck-Hodge index theorem (see \cite{grothendieck}), which says that $\mathcal{H}$ has signature $(1,\rho -1)$. 

\subsection{Some norms on the vector spaces $N^p_{\mathbb{R}}(X)$ and $N^p_{\mathbb{C}}(X)$}

Given $\iota :X\subset \mathbb{P}^N_K$ a projective manifold of dimension $k$, let $H\in A^1(\mathbb{P}^N)$ be a hyperplane and $\omega _X=H|_{X}=\iota ^*(H)\in A^1(X)$. For an irreducible subvariety $V\subset X$ of codimension $p$, we define the degree of $V$ to be $\deg (V)=$ the degree of the dimension $0$ cycle $V.\omega _X^{k-p}$, or equivalently $\deg (V)=$degree of the variety $\iota _*(V)\subset \mathbb{P}^N$. Similarly, we define for an effective codimension $p$ cycle $V=\sum _{i}m_i[V_i]$ (here $m_i\geq 0$ and $V_i$ are irreducible), the degree $\deg (V)=\sum _im_i\deg (V_i)$. This degree is extended to vectors in $N^p_{\mathbb{R}}(X)$. Note that the degree map is a numerical equivalent invariant.   

As a consequence of Roberts'  version of Chow's moving lemma, we have the following result on intersection of cycles.

\begin{lemma}
Let $V$ and $W$ be irreducible subvarieties in $X$. Then the intersection $V.W\in A^*(X)$ can  be represented as $V.W=\alpha _1-\alpha _2$, where $\alpha _1,\alpha _2\in A^*(X)$ are effective cycles and $\deg (\alpha _1),\deg (\alpha _2)\leq C\deg (V)\deg (W)$, where $C>0$ is a constant independent of $V$ and $W$. 
\label{LemmaDegreeOfIntersections}\end{lemma}
\begin{proof}
Using Roberts' version of Chow's moving lemma, $W$ is rationally equivalent to 
\begin{eqnarray*}
W'=\sum _{i=1}^e(-1)^{i-1}gC_{L_i}(W_{i-1}).X+(-1)^eW_e,
\end{eqnarray*}
where $W_0=W$, $W_i=C_{L_i}(W_{i-1}).X-W_{i-1}$, $C_{L_i}(W_{i-1})\subset \mathbb{P}^N_K$ is a cone over $W_{i-1}$, and $g\in Aut(\mathbb{P}^N_K)$ is an automorphism. Moreover, $gC_{L_i}(W_{i-1}).X$, $gC_{L_i}(W_{i-1}).V$ and $W_e.V$ are all well-defined. We note that $e\leq k=\dim (X)$, and for any $i=1,\ldots ,e$ 
\begin{eqnarray*} 
\deg (W_i)&\leq& \deg (gC_{L_i}(W_{i-1}).X)\leq \deg (gC_{L_i}(W_{i-1}))\deg (X)\\
&=&\deg (C_{L_i}(W_{i-1})).\deg (X)=\deg (W_{i-1})\deg (X).
\end{eqnarray*}
Here we used that $\deg (C_{L_i}(W_{i-1})=\deg (W_{i-1})$ (see Example 18.17 in \cite{harris}), and $\deg (gC_{L_i}(W_{i-1})=\deg (C_{L_i}(W_{i-1})$ because $g$ is an automorphism of $\mathbb{P}^N$ (hence a linear map).  

Therefore, the degrees of $W_i$ are all $\leq (\deg (X))^k\deg (W)$. By definition, the intersection product $V.W\in A^*(X)$ is given by $V.W'$, which is well-defined. We now estimate the degrees of each effective cycle $gC_{L_i}(W_{i-1})|_X.V$ and $W_e.V$. Firstly, we have by the projection formula
\begin{eqnarray*}
\deg (gC_{L_i}(W_{i-1})|_X.V)&=&\deg (\iota _*(gC_{L_i}(W_{i-1})|_X.V))=\deg (gC_{L_i}(W_{i-1}).\iota _*(V))\\
&=&\deg (C_{L_i}(W_{i-1})).\deg (V)\leq \deg (X)^k\deg (W)\deg (V).
\end{eqnarray*}  
Finally, we estimate the degree of $W_e.V$. Since $W_e.V$ is well-defined, we can choose a linear subspace $L\subset \mathbb{P}^N$ so that $C_L(W_e).X$ and $C_L(W_e).V$ are well-defined. Recall that $C_L(W_e)-W_e$ is effective, we have
\begin{eqnarray*}
\deg (V.W_e)\leq \deg (V.C_L(W_e)|_{X})=\deg  (V).\deg (C_L(W_e))\leq \deg (X)^k\deg (V)\deg (W). 
\end{eqnarray*}

From these estimates, we see that we can write
\begin{eqnarray*}
V.W'=\alpha _1-\alpha _2,
\end{eqnarray*}
where $\alpha _1,\alpha _2$ are effective cycles and $\deg (\alpha _1),\deg (\alpha _2)\leq C\deg (V)\deg (W)$, where $C=k.\deg (X)^k$ is independent of $V$ and $W$.
\end{proof}

Using this degree map, we define for an arbitrary vector $v\in N^p_{\mathbb{R}}(X)$, the norm 
\begin{equation}
\|v\| _1=\inf \{\deg (v_1)+\deg (v_2):~v=v_1-v_2,~v_1, v_2\in N^p_{\mathbb{R}}(X) \mbox{ are effective}\}.
\label{LabelDefinitionNorm}\end{equation}
That this is actually a norm follows easily from Lemma \ref{LemmaDegreeOfIntersections} and that the bilinear form $N^p(X)\times N^{k-p}(X)\rightarrow \mathbb{Z}$, $(v,w)\mapsto \deg (v.w)$ is non-degenerate. If $v\in N^p_{\mathbb{R}}(X)$ is effective, then $\|v\|_1=\deg (v)$. Since $N^p_{\mathbb{R}}(X)$ is of finite dimensional, any norm on it is equivalent to $\|\cdot \|_1$. We can also complexify these norms to define norms on $N^p_{\mathbb{C}}(X)$.

\section{Correspondences}

In this section, we present results on correspondences  which are needed to prove Theorem \ref{TheoremRelativeDynamicalDegrees}. We will be mostly concerned with correspondences between quasi-projective varieties over an algebraically closed field $K$. These assumptions allow us to use the following result (Theorem 1.25, Section I.6.3 in \cite{shafarevich}) on dimensions of fibers of surjective regular morphisms. 
\begin{theorem}
Let $K$ be an algebraically closed field, $X,Y$ irreducible quasi-projective varieties over $K$ and $f:X\rightarrow Y$ a surjective regular morphism. Then, there exists a non-empty Zariski open subset $U\subset Y$ such that $\dim (f^{-1}(y))=\dim (X)-\dim (Y)$ for all $y\in U$. 
\label{TheoremFiber}\end{theorem}
In this paper we will need to consider various intersection of varieties, and the condition for these to be well-defined is that the set theoretical intersection has the correct dimension. Theorem \ref{TheoremFiber} helps us to check that condition is satisfied over a non-empty Zariski open set. We do not need the stronger version of it, that is the Sard's theorem, which is available only in characteristic zero.

\subsection{Correspondences}

Let $X$ and $Y$ be quasi-projective varieties of pure dimensions $k$ and $l$. 

First, we assume that $X,Y$ are irreducible. A correspondence $f:X\rightarrow Y$ is represented by its graph $\Gamma _f=\sum _{i=1}^m\Gamma _i$, viewed as an algebraic cycle, where each $\Gamma _i$ is an irreducible subvariety of dimension $k$ of $X\times Y$. We do not assume that the varieties $\Gamma _i$ are distinct. We say that $f$ is irreducible if $\Gamma _f=a\Gamma$ where $a\in \mathbb{N}$ and $\Gamma$ is irreducible. A correspondence is dominant if for each $i$, the two natural projections from $\Gamma _i$ to $X,Y$ are dominant. For example, if $f:X\rightarrow Y$ is a dominant rational map then it is an irreducible dominant correspondence. If $f$ is a correspondence whose graph is $\Gamma _f=\sum _i\Gamma _i$, we denote by $af$ the correspondence whose graph is $\Gamma _{af}=\sum _ia\Gamma _i$. 

A correspondence $f:X\rightarrow Y$ is kind of a multi-value map, which for each $x\in X$ gives $f(x)=\Gamma _{f,x}$ the projection to $Y$ of points in $\Gamma _f$ whose first coordinate is $x$. We need also take care of multiplicities: if several points in $\Gamma _{f}$ (with the same first coordinate $x$) maps to the same point $y\in Y$, then in $f(x)$ we need to list $y$ with the corresponding multiplicity.  

Dominant correspondences can be composed, much as the compositions between dominant rational maps, and the resulting correspondence is also dominant. In fact, this can be done as follows. Let $f:X\rightarrow Y$ and $g:Y\rightarrow Z$ be two dominant correspondences. By Theorem \ref{TheoremFiber}, there are non-empty Zariski open subsets $U\subset X,U'\subset Y$ over which the fibers of the projections $\Gamma _f\rightarrow X$, $\Gamma _g\rightarrow Y$ have the correct dimension, and $f(U)\subset U'$. Then we define for $x\in U$: $g\circ f(x)=g(f(x))$. This is a correspondence from $U$ to $Z$,  taking the closure of the graph of this in $X\times Z$ we obtain the graph of the correspondence $g\circ f:X\rightarrow Z$.  

Given two dominant correspondences $f:X\rightarrow X$ and $g:Y\rightarrow Y$, we say that a dominant rational map $\pi :X\rightarrow Y$ is a semi-conjugacy of $f$ and $g$ if $\pi \circ f=g\circ \pi$, and express this by writing $\pi :(X,f)\rightarrow (Y,g)$. By the definition of compositions of correspondences, we see that $\pi \circ f=g\circ \pi$ iff there is a non-empty Zariski open set $U\subset X$ such that $\pi \circ f(x)=g\circ \pi (x)$ for all $x\in U$. Since the composition of correspondences are associative, we obtain the following result.
\begin{lemma}
 If $(X,f)\rightarrow (Y,g)$ and $(Y,g)\rightarrow (Z,h)$ are semi-conjugate, then the induced $(X,f)\rightarrow (Z,h)$ is also semi-conjugate. 
\label{LemmaSemiconjugacyAssociativity}\end{lemma}
In particular, if $\pi :(X,f)\rightarrow (Y,g)$ is a semi-conjugacy, then $\pi \circ f^n=g^n\circ \pi$ for all $n\in \mathbb{N}$. This fact, however, is not true in general if we consider correspondences over reducible varieties, see below. 

The above definitions can be extended to the case where $X$ and $Y$ are reducible. In this case, let $X_1,\ldots ,X_m$ be components of $X$ and $Y_1,\ldots ,Y_p$ components of $Y$. Then a correspondence between $X$ and $Y$ is a collection of correspondences $f_{i,j}:X_i\rightarrow Y_j$, $(i,j)\in I$ where $I\subset \{1,\ldots ,m\}\times \{1,\ldots ,p\}$. The correspondence is dominant iff each $f_{i,j}$ is dominant, the union of all the domains of the collection is dense in $X$, and the union of all the ranges of the collection is dense in $Y$. If $f:X\rightarrow X$ is a dominant correspondence, then we can iterate $f$ in an obvious way. More precisely, if $f_{i,j}$ and $f_{k,l}$ are two irreducible component of $f$, we compose $f_{i,j}f_{k,l}$ whenever the range of $f_{k,l}$ is the domain of $f_{i,j}$, otherwise we do not compose them.   

A semi-conjugacy between two dominant correspondences $(X,f)$ and $(Y,g)$ is then a dominant rational map $\pi :X\rightarrow Y$, which in this case means a collection of dominant rational maps $\pi _i:X_i\rightarrow Y_{a(i)}$ for $i=1,\ldots ,m$ such that $Y=\bigcup _{i}Y_{a(i)}$ (hence each $X_i$ is in the domain of exactly one rational map $\pi _i$, while each $Y_i$ may be in the range of different such maps), satisfying  $\pi \circ f^n=g^n\circ \pi $ for all $n\in \mathbb{N}$.

{\bf Remark.} There is a subtlety in defining semi-conjugacy for correspondences on reducible varieties, as the readers may notice. In this general case, by the very definition of the composition of correspondences, if we have $\pi \circ f=g\circ \pi$, it is not assured that we also have $\pi \circ f^2=g^2\circ \pi$. For example, assume that $X=X_1\cup X_2$ has two irreducible components,  $Y$ is irreducible, and $\pi $ is a dominant rational map $\pi =(\pi _1,\pi _2):(X_1,X_2)\rightarrow Y$. Assume that there are correspondences $f_{1,2}:X_1\rightarrow X_2$, $f_{2,2}:X_2\rightarrow X_2$ and $f_{2,1}:X_2\rightarrow X_1$ each semi-conjugate to the same correspondence $(Y,g)$. Then it can be checked that $f=2f_{1,2}+f_{2,1}+f_{2,2}:X\rightarrow X$ is semi-conjugate to $2g$. Hence, $f^2=2f_{1,2}\circ f_{2,2}+2f_{2,1}\circ f_{1,2}+f_{2,1}\circ f_{2,2}$ while $(2g)^2=4g^2$. From this, we can check that $\pi\circ  f^2\not= g^2\circ \pi$. 

To work with semi-conjugacies of correspondences over reducible varieties, it is convenient to consider semi-conjugacies of pairs. Let $f:X_1\rightarrow X_2$ and $g:Y_1\rightarrow Y_2$ be dominant correspondences between irreducible varieties, where $\dim (X_1)=\dim (X_2)$ and $\dim (Y_1)=\dim (Y_2)$.  A semi-conjugacy between $(X_1,X_2,f)$ and $(Y_1,Y_2,g)$ is then a pair of dominant rational maps $\pi _1:X_1\rightarrow Y_1$ and $\pi _2:X_2\rightarrow Y_2$ such that $\pi _2\circ f=g\circ \pi _1$. We describe this fact by writing $(\pi _1,\pi _2):(X_1,X_2,f)\rightarrow (Y_1,Y_2,g)$.

The following result characterises semi-conjugate correspondences of pairs. 
\begin{lemma}
Let $(\pi _1,\pi _2):(X_1,X_2,f)\rightarrow (Y_1,Y_2,g)$ be a semi-conjugacy, where $\dim (X_1)=\dim (X_2)$ and $\dim (Y_1)=\dim (Y_2)$. Write $\Gamma _g=\sum _{i}a_iG_i $, where $G_i$ are distinct irreducible varieties and $a_i\in \mathbb{N}$ is the multiplicity of $G _i$. For a generic $x\in X$, the cardinality of $f(x)\cap \pi _2^{-1}(z)$ is $a_i$, for every $z\in g\pi _1(x)$ such that $(\pi _1(x),z)\in G _i$. 
 \label{LemmaSemiconjugacyConstantProperty}\end{lemma}
\begin{proof}
Let $x\in X$ be generic. Define $y=\pi _1(x)$. Then by assumption, we have
\begin{eqnarray*}
\pi _2f(x)=g\pi _1(x)=g(y).
\end{eqnarray*}
If $(y,z)\in G _i$ then $z$ appears in $g(y)$ with multiplicity $a_i$. From this, it follows that there are exactly $a_i$ points (counted with multiplicities) in $f(x)$ which is mapped to $z$ by the map $\pi _2$. In other words, $\sharp f(x)\cap \pi _2^{-1}(z)=a_i$. 
\end{proof}


\subsection{Pullback of correspondences by rational maps}

In this section we define certain pullback operators of correspondences by dominant rational maps. Case 2 is not needed in the sequel, since in general the pullback correspondence is over a reducible variety, but we include it here since it seems interesting by itself. 

{\bf Case 1:} Equi-dimensional pullback. 

Let $\pi =(\pi _1,\pi _2):(X_1,X_2)\rightarrow (Y_1,Y_2)$ be a dominant rational map of pairs, and $g:Y_1\rightarrow Y_2$  a dominant correspondence. Assume that $\dim (X_1)=\dim (X_2)=\dim (Y_1)=\dim (Y_2)$. The pullback of $g$ by $\pi$, denoted by  $f=\pi ^*g:X_1\rightarrow X_2$ is defined as follows: for $U\subset X_1$ a suitable non-empty Zariski open set, we define $f(x_1)=\pi _2^{-1}g\pi _1(x_1)$. The graph of $f$ is then the strict pullback  of the graph of $g$ by the generically finite rational map $\pi _1\times \pi _2:X_1\times X_2\rightarrow Y_1\times Y_2$. Since $X_1,X_2,Y_1,Y_2$ are of the same dimension, the cardinality of $f(x_1)$ is finite.   
\begin{lemma}
Let $a=$ the degree of $\pi _2$, that is the cardinality (counted with multiplicities) of a generic fiber of $\pi _2$. We have $\pi _2\circ f=a(g\circ \pi _1)$. In other words, we have a semi-conjugacy $\pi :(X_1,X_2,f=\pi ^*g)\rightarrow (Y_1,Y_2,ag)$.
\label{LemmaSemiconjugacy1}\end{lemma}  
\begin{proof}
For a generic point $x_1\in X_1$, put $y_1=\pi _1(x_1)$ we have by definition
\begin{eqnarray*}
\pi _2\circ f(x_1)=\pi _2\circ \pi _2^{-1}g(y_1)=a (g(y_1))=ag\circ \pi _1(x_1),
\end{eqnarray*} 
as wanted. 
\end{proof}

{\bf Case 2:} Pullback of a triple of correspondences. This is more subtle than the pullback of a triple of maps. 

Let be given two dominant semi-conjugacies $\pi :(X_1,X_2,f)\rightarrow (Y_1,Y_2,g)$ and $\psi :(W_1,W_2,h)\rightarrow (Y_1,Y_2,g)$, where $X_1,X_2,Y_1,Y_2,W_1,W_2$ are of pure dimensions, $\dim (X_1)=\dim (X_2)$, $\dim (Y_1)=\dim (Y_2)$ and $\dim (W_1)=\dim (W_2)$. We can pullback this triple to a dominant correspondence $u:Z_1\rightarrow Z_2$ for which there are semi-conjugacies $(Z_1,Z_2,u)\rightarrow (X_1,X_2,af)$ and $(Z_1,Z_2,u)\rightarrow (W_1,W_2,ah)$ where $a$ is a positive integer depending only on the correspondence $g:Y_1\rightarrow Y_2$. Note that in general $Z_1$ and $Z_2$ may be reducible. 

The following construction, which is done under the additional assumption that $X_1=X_2=X$, $Y_1=Y_2=Y$ and $W_1=W_2=W$, extends easily to the general case. 

- Defining $Z$. On an appropriate non-empty Zariski open sets $U\subset X$ and $U'\subset W$, the (reduced) variety $\{(x,w)\in U\times U':~\pi (x)=\psi (w)\}$ is of pure dimension $\dim (X)+\dim (W)-\dim (Y)$. Let $Z$ be the closure of this variety in $X\times W$.  In other words, $Z$ is the usual pullback of the triple $\pi :X\rightarrow Y$ and $\psi:W\rightarrow Y$. We note that in general $Z$ is not irreducible. 

We denote the natural projections from $Z$ to $X,Y$ by $\varphi :Z\rightarrow X$ and $\pi _1:Z\rightarrow W$.

- Defining the correspondence $u:Z\rightarrow Z$.

Let $\Gamma $ be the closure in $X\times W\times X\times W$ of the set $$\{(x,w,x',w'):~(x,x')\in V,(w,w')\in V',~\pi (x)=\psi (w), ~\pi (x')=\psi (w')\},$$
where $V$ is an appropriate Zariski open set of $\Gamma _f$ and $V'$ is an appropriate Zariski open set of $\Gamma _{h}$. We can choose for example $V=\{(x,x')\in \Gamma _f, x\in U\}$ for an appropriate non-empty Zariski open set $U\subset X$, and similarly  $V'=\{(w,w')\in \Gamma _{h},~w\in U'\}$ for some appropriate non-empty Zariski open set $U'\subset W$. 

We have the following properties of $\Gamma$.

a) $\Gamma \subset Z\times Z$: This is clear from the definition. 

b) $\Gamma$ is non-empty and of pure dimension $=$ $\dim (Z)$: Let $(x,w)\in U\times U'$, and $y=\pi (x)=\psi (w)$. Then, from
\begin{eqnarray*}
\pi f(x)=g(y)=\psi h(w),
\end{eqnarray*}
thanks to the two semi-conjugacies $(X,f)\rightarrow (Y,g)$ and $(W,h)\rightarrow (Y,g)$, it follows that for any $x'\in f(x)$ there is $w'\in h(w)$ such that $\pi (x')=\psi (w')$, and vice versa. Thus $\Gamma$ is non-empty. Moreover, since both $f(x)$ and $h(w)$ have bounded cardinalities, it follows that $\Gamma $ has dimension $=$ $\dim (Z)$.  

$\Gamma$ is the support for the graph $\Gamma _{u}$ of the correspondence $u$ which we sought to define. We will now need to  choose appropriately the multiplicities of the irreducible components of $\Gamma$. 

Let $\Gamma _g=\sum _{i\in I}a_iG _i$, where $I$ is a finite set, $\Gamma _i$ are distinct irreducible varieties and $a_i\in \mathbb{N}$. We define $a(g):=\prod _{i\in I}a_i$. Let $V$ be an irreducible component of the (reduced) variety  $\Gamma$. Let $\tau :Z\rightarrow Y$ be $\tau =\psi \circ \pi _1=\pi \circ \varphi$.  Then the (strict) image of $V$ (see Section 4 for definition) under the rational map $\tau \times \tau :Z\times Z\rightarrow Y\times Y$ must lie in exactly one of the components $G_i$ of $\Gamma _g$. Then, we define $b(V):=a_i$. 

Now we define 
\begin{eqnarray*}
\Gamma _{u}:=\sum _Va(g)V/b(V),
\end{eqnarray*} 
where $V$ runs all over the irreducible components of $\Gamma$. Since $a(g)=\prod _{i\in I}a_i$ and $b(V)$ is one of the $a_i$'s, it follows that $a(g)/b(V)$ are integers for all $V$. From the properties of $\Gamma$, it follows immediately that $\Gamma _{u}$ is the graph of a dominant correspondence $u:Z\rightarrow Z$.  

\begin{lemma}
Let $G_i$, $a_i$, $a(g)$ and $a(V)$ be as in the previous paragraph. We have $\varphi \circ u=a(g)(f\circ \varphi ) $ and $\pi _1\circ u=a(g)(h\circ \pi _1)$. In other words, we have two semi-conjugacies $\varphi :(Z,u)\rightarrow (X,a(g)f)$ and $\pi _1:(Z,u)\rightarrow (W,a(g)h)$.  
\label{LemmaSemiconjugacy2}\end{lemma}
\begin{proof}
Let $z=(x,w)\in Z\cap (U\times U')$, where $U\subset X$ and $U'\subset W$ are appropriate non-empty Zariski open sets. Define $y=\pi (x)=\psi (w)$. We have by definition
\begin{eqnarray*}
u(z)=\bigcup _{t\in g(y)}a(g)/a(t)(f(x)\cap \pi ^{-1}(t))\times (h(w)\cap \psi ^{-1}(t)),
\end{eqnarray*}
where $a(t)=a_i$ provided $(y,t)\in G_i$. By Lemma \ref{LemmaSemiconjugacyConstantProperty}, we have $\sharp (f(x)\cap \pi ^{-1}(t))=a(t)=\sharp (h(w)\cap \psi ^{-1}(t))$ for all $t\in g(y)$. Hence
\begin{eqnarray*}
\pi _1\circ u(z)&=&\bigcup _{t\in g(y)}a(g)/a(t)\times \sharp (f(x)\cap \pi ^{-1}(t))\times (h(w)\cap \psi ^{-1}(t))\\
&=&a(g)\bigcup _{t\in g(y)} (h(w)\cap \psi _1^{-1}(t))\\
&=&a(g)h(w)=a(g)h\pi _1(z).
\end{eqnarray*}
This shows that we have a semi-conjugacy $\pi _1:(Z,u)\rightarrow (W,a(g)h)$. Similarly, we have a semi-conjugacy $\varphi :(Z,u)\rightarrow (X,a(g)f)$. 
\end{proof} 

\section{Pullback by correspondences and strict intersections of varieties}

In this section we consider pullback, strict transforms by correspondences and strict intersection of algebraic varieties. We assume throughout this section that the field $K$ is algebraic closed of characteristic zero.

\subsection{Strict intersection of varieties} In this subsection, we prove an estimate for strict intersection of algebraic cycles on a product space $X\times Y$. We first define the notion of strict intersection of two varieties. Let $Z$ be an irreducible smooth projective variety and $V,W\subset Z$ be irreducible varieties. Let $U\subset Z$ be a non-empty Zariski open set on which the set theoretical intersection $V\cap W$ is of the correct dimension. Then $V|_U\wedge W|_U$ is well-defined as a variety, and we define $V\stackrel{U}{\wedge}W$ as the closure in $Z$ of $V|_U\wedge W|_U$. 

The main result of this subsection is the following.
\begin{lemma}
Let $X$ and $Y$ be irreducible smooth projective varieties over $K$. Let $\omega _X$ and $\omega _Y$ be generic hyperplane sections on $X$ and $Y$. There exists a constant $A>0$ such that for any pair of irreducible subvarieties $V,W\subset X\times Y$ and an open subset $U\subset X\times Y$ over which the wedge intersection $V|_U{\wedge}W|_U$ is well defined, then in $N^*(X\times Y)$
\begin{eqnarray*}
W\stackrel{U}{\wedge}V\leq (A\sum _{0\leq j\leq p(W), ~0\leq p(W)-j\leq k-l}\alpha _j(W)\omega _X^j.\omega _Y^{p(W)-j}).[V]
\end{eqnarray*}
 Here $p(W)$ is the codimension of $W$ and $\alpha _j(W)=W.\omega _X^{k-p(W)+j}.\omega _Y^{l-j}$.
\label{LemmaChowMovingForProduct}\end{lemma}
 \begin{proof}
We will follow the ideas in Section 2 in \cite{dinh-nguyen-truong1} plus that in \cite{roberts}. Let us define $Z=X\times Y$. Let $X,Y$ be embedded into the same projective space $\mathbb{P}^N$. Let $\Delta _Z\subset Z\times Z$ be the diagonal, and let $\pi _1,\pi _2:Z\times Z\rightarrow Z$ be the projections. For simplicity, let us use  $\omega _X$ and $\omega _Y$ to denote also the pullbacks,  via the projections $Z=X\times Y\rightarrow X,Y$, of $\omega _X$ and $\omega _Y$ to $Z$. We embed $\mathbb{P}^{N}\times \mathbb{P}^N$ into some $\mathbb{P}^{N'}$, and hence have an embedding of $Z\times Z$ into $\mathbb{P}^{N'}$. 

By assumption, the two varieties $\Delta _Z$ and $\pi _1^{-1}(V)\wedge \pi _2^{-1}(W)$ has well-defined wedge intersection on the Zariski open set $U_Z=\pi _1^{-1}(U)\cap \pi _2^{-1}(U)\subset Z\times Z$. Hence in terms of the notation in \cite{roberts}, the excess of $\Delta _Z$ with relative to $\pi _1^{-1}(V)|_U\wedge \pi _2^{-1}(W)|_U$ is $0$.

Applying the Main Lemma in \cite{roberts} (note that in there the varieties $Y_1,\ldots ,Y_m$ are not required to be closed), where $V$ in that lemma is replaced by $Z\times Z$; $Z$ in that lemma is replaced by $\Delta _Z$; and the $Y_1,\ldots ,Y_m$ in that lemma are replaced by $\pi _1^{-1}(V)|_U\wedge \pi _2^{-1}(W)|_U$ (note that this variety is not closed in $Z\times Z$) in our situation, it follows that for a generic choice of a linear subspace $L\subset \mathbb{P}^{N'}$ of appropriate dimension, the cone $C_L(\Delta _Z)\subset \mathbb{P}^{N'}$ intersects properly $Z\times Z$, and $C_L(\Delta _Z)|_{Z\times Z}$ intersects properly $\pi _1^{-1}(V)|_U\wedge \pi _2^{-1}(W)|_U$. 

Since $\Delta _Z$ is a component of $C_L(\Delta _Z)\cap (Z\times Z)$, it follows that 
\begin{eqnarray*}
\Delta _Z\wedge \pi _1^{-1}(V)|_U\wedge \pi _2^{-1}(W)|_U\leq C_L(\Delta _Z)|_{Z\times Z} \wedge \pi _1^{-1}(V)|_U\wedge \pi _2^{-1}(W)|_U.
\end{eqnarray*}

There is a $\mathbb{P}^1$-family $g(t)$ of automorphisms of $\mathbb{P}^{N'}$ with $g(0)=Id$ such that $g(t)C_L(\Delta _Z)$ intersects $Z\times Z$ properly and $g(t)C_L(\Delta _Z)|_{Z\times Z}$ intersects $\pi _1^{-1}(V)\wedge \pi _2^{-1}(W)$ properly, for all $t$ in a non-empty Zariski open set $\Omega \subset \mathbb{P}^1$. Let $\gamma (t)$ be the closure of the cycles $g(t)C_L(\Delta _Z)|_{Z\times Z}\wedge \pi _1^{-1}(V)\wedge \pi _2^{-1}(W)$ for $t\in \Omega$ inside $Z\times Z\times \mathbb{P}^1$. Then 
\begin{eqnarray*}
\gamma (0)\geq   C_L(\Delta _Z)|_{Z\times Z} \wedge \pi _1^{-1}(V)|_U\wedge \pi _2^{-1}(W)|_U\geq \Delta _Z\wedge \pi _1^{-1}(V)|_U\wedge \pi _2^{-1}(W)|_U,
\end{eqnarray*}
and hence 
\begin{eqnarray*}
(\pi _1)_*(\gamma (0))\geq (\pi _1)_*(\pi _1^{-1}(V)|_U\wedge \pi _2^{-1}(W)|_U\wedge \Delta _Z)=V|_U\wedge W|_U.
\end{eqnarray*}
Since $\gamma (0)$ is a closed subvariety, it then follows that $(\pi _1)_*(\gamma (0))\geq V\stackrel{U}{\wedge }W$. 

Hence, to finish the proof it is enough to bound the class of $(\pi _1)_*(\gamma (0))$ in $N^*(Z)$. The latter is the same as the class of $\gamma (t)$ for $t\in \Omega$. Let $H\subset \mathbb{P}^{N'}$ be a hyperplane, then the class of $C_L(\Delta _Z)$ is $AH^{k+l}$ for some positive integer $A>0$. For $t\in \Omega$, since $g (t)$ is a linear map thus preserves degree, we have in $N^*(Z\times Z)$
\begin{eqnarray*}
[\gamma (t)]&=&[g(t)]C_L(\Delta _Z)]|_{Z\times Z}. [\pi _1^{-1}(V)]. [\pi _2^{-1}(W)]\\
&=&A[H^{k+l}]|_{Z\times Z}.[\pi _1^{-1}(V)]. [\pi _2^{-1}(W)]\\\
&=&A.(\pi _1^{-1}\omega _X+\pi _1^{-1}(\omega _Y)+\pi _2^{-1}(\omega _X)+\pi _2^{-1}\omega _Y)^{k+l}.[\pi _2^{-1}(W)].[\pi _1^{-1}(V)].
\end{eqnarray*}
The push forward by $\pi _1$ of the lass expression is bounded by
\begin{eqnarray*}
(A'\sum _{0\leq j\leq p(W), ~0\leq p(W)-j\leq k-l}\alpha _j(W)\omega _X^j.\omega _Y^{p(W)-j}).[V],
\end{eqnarray*}
where $A'$ depends only on $A$ and the binomial numbers $C(k+l,i)$ ($i=0,\ldots ,k+l$), taking into account that $\dim (X)=k$ and $\dim (Y)=l$. Hence the lemma is proven. 

 \end{proof}

\subsection{Pullback and strict transforms of algebraic cycles by correspondences}
Let $X$ and $Y$ be two irreducible projective manifolds and $f:X\rightarrow Y$ a dominant correspondence. Let $\Gamma _f$ be the graph of $f$, and write $\Gamma _f=\sum _{i\in I}a_iG_i$ where $I$ is a finite set, $G_i$ are distinct irreducible varieties and $a_i\in \mathbb{N}$.

Then we can define the pushforward operators $f_*:A_q(X)\rightarrow A_q(Y)$ and pullback operators $f^*:A^p(Y)\rightarrow A^p(X)$ (see Chapter 16 in \cite{fulton}). For example, there are two methods to define the pullback operators: 

Method 1: Let $\pi _X,\pi _Y:X\times Y\rightarrow X,Y$ be the two projections, and let $\Gamma _f$ be the graph of $f$. For $\alpha \in A^p(Y)$, we define $f^*(\alpha )\in A^p(X)$ by the following formula
\begin{eqnarray*}
f^*(\alpha )=(\pi _X)_*(\Gamma _f.\pi _Y^*(\alpha ))=\sum _{i\in I}a_iG_i^*(\alpha ). 
\end{eqnarray*}

Method 2: Let $\tau _i:\Gamma _i\rightarrow G _i$ be an alteration of $G _i$, that is $\Gamma _i$ is smooth and $\tau _i$ is generically finite of degree $d_i$. The results of de Jong assure the existence of alterations, and let $p_i,g_i:\Gamma _i\rightarrow X,Y$ be the induced morphisms. Then we define 
\begin{eqnarray*}
f^*(\alpha )=\sum _{i\in I}a_i(p_i)_*(g_i^*(\alpha ))/d_i. 
\end{eqnarray*}

For the convenience of the readers, we recall here an argument to show why these two methods give the same answer. It suffices to show for each $i$ that
\begin{eqnarray*}
G_i^*(\alpha )=(p_i)_*(g_i^*(\alpha ))/d_i.
\end{eqnarray*}
In fact, since $p_i=\pi _X\circ \tau _i$ and $g_i=\pi _Y\circ \tau _i$ we have
\begin{eqnarray*}
(p_i)_*g_i^*(\alpha )=(\pi _X)_*(\tau _i)_*\tau _i^*\pi _Y^*(\alpha )=d_i(\pi _X)_*(\pi _Y^*(\alpha ).G_i).
\end{eqnarray*}
Here we use that by the projection formula: $(\tau _i)_*\tau _i^*(\beta )=d_iG_i.\beta $ for any algebraic cycle $\beta$ on $X\times Y$.  

In defining dynamical degrees and proving some of their basic properties, we need to estimate the degrees of the pullback and of strict transforms by a correspondence of a cycle. We present these estimates in the remaining of this subsection. We fix an alteration $\Gamma _i$ of the component  $G _i$ of $\Gamma _f$, and let $p_i,g_i:\Gamma _i\rightarrow X,Y$ be the induced morphisms. By Theorem \ref{TheoremFiber}, the sets
\begin{eqnarray*} 
V_{l}=\{y\in Y:~\dim (g_i^{-1}(y))\geq l\}
\end{eqnarray*}
are algebraic varieties of $Y$. We denote by $\mathcal{C}_{g_i}=\cup _{l>\dim (X)-\dim (Y)}V_l$ the critical image of $g_i$. We have the first result considering the pullback of a subvariety of $Y$.

\begin{lemma}
Let $W$ be an irreducible subvariety of $Y$. If $W$ intersects properly every irreducible component of $V_l$ (for any $l>\dim (X)-\dim (Y)$), then $g_i^{-1}(W)$ is well-defined as a subvariety of $\Gamma _i$, and moreover the class $[g_i^{-1}(W)]$ represents the pullback $g_i^*(W)$ in $A^*(\Gamma _i)$.  
\label{LemmaGoodPullbackVariety}\end{lemma}
\begin{proof} (See also Example 11.4.8 in \cite{fulton}.) By the intersection theory (see Section 8.2 in \cite{fulton} and Theorem 3.4 in \cite{friedlander-lawson}), it suffices to show that $g_i^{-1}(W)$ has the correct dimension $\dim (X)-\dim (Y)+\dim (W)$. First, if $y\in W-\mathcal{C}_{g_i}$ then $\dim (g_i^{-1}(y))=\dim (X)-\dim (Y)$ by definition of $\mathcal{C}_{g_i}$. Hence $\dim (g_i^{-1}(W-\mathcal{C}_{g_i})=\dim (W)+\dim (X)-\dim (Y)$. It remains to show that $g_i^{-1}(W\cap \mathcal{C}_{g_i})$ has dimension $\leq \dim (X)+\dim (W)-\dim (Y)-1$. Let $Z$ be an irreducible component of $W\cap \mathcal{C}_{g_i}$. We define $l=\inf \{\dim (g_i^{-1}(y)):~y\in Z\}$. Then $l>\dim (X)-\dim (Y)$ and for generic $y\in Z$ we have $\dim (g^{-1}(y))=l$ by Theorem \ref{TheoremFiber}. Let $V\subset V_l$ be an irreducible component containing $Z$.  By assumption $V\cap W$ has dimension $\dim (V)+\dim (W)-\dim (Y)$, hence $\dim (Z)\leq \dim (V)+\dim (W)-\dim (Y)$. We obtain $$\dim (g_i^{-1}(Z-V_{l+1}))=l+\dim (Z)\leq l+\dim (V)+\dim (W)-\dim (Y).$$ Since $g_i$ is surjective (because $f$ is dominant) and $V\not= Y$, it follows that $$\dim (X)-1\geq \dim (g_i^{-1}(V))\geq \dim (V)+l.$$ From these last two estimates we obtain 
\begin{eqnarray*}
\dim (g_i^{-1}(Z-V_{l+1}))&=&l+\dim (V)+\dim (W)-\dim (Y)\\
&\leq& \dim (X)-1+\dim (W)-\dim (Y).
\end{eqnarray*}
Since there are only a finite number of such components, it follows that $\dim (g_i^{-1}(W\cap \mathcal{C}_{g_i}))\leq \dim (W)+\dim (X)-\dim (Y)-1$, as claimed. 
\end{proof} 
 
We next estimate the degree of the pullback of a cycle. Fix an embedding $Y\subset \mathbb{P}^N_K$, and let $\iota :Y\subset \mathbb{P}^N_K$ the inclusion. Let $H\subset \mathbb{P}^N_K$ be a generic hyperplane and let $\omega _Y=H|_{Y}$. 
 
\begin{lemma} Let $f:X\rightarrow Y$ be a correspondence. 

a) Let $p=0,\ldots ,\dim (Y)$, and $Z\subset X$ a proper subvariety. Then there is a linear subspace $H^p\subset \mathbb{P}^N_K$ of codimension $p$ such that $H^p$ intersects $Y$ properly, $f^*(\iota ^*(H^p))$ is well-defined as a subvariety of $X$, and $f^*(\iota ^*(H^p))$ has no component on $Z$. In particular, for any non-negative integer $p$, the pullback $f^*(\omega _Y^p)\in A^{p}(X)$ is effective. 

b) Let $W$ be an irreducible of codimension $p$ in $Y$. Then in $A^p(X)$, we can represent $f^*(W)$ by $\beta _1-\beta _2$, where $\beta _1$ and $\beta _2$ are effective and $\beta _1,\beta _2\leq C\deg (W)f^*(\omega _Y^p)$ for some constant $C>0$ independent of the variety $W$, the manifold $X$ and the correspondence $f$.
\label{LemmaDegreeOfPullback}\end{lemma} 
\begin{proof} We may assume that $\Gamma _f$ is irreducible. Let $\Gamma \rightarrow \Gamma _f$ be an alteration, with the induced morphisms $p:\Gamma \rightarrow X$ and $g:\Gamma \rightarrow Y$.  Since by definition $f^*(W)=p_*g^*(W)$ and since $p_*$ preserves effective classes, it suffices to prove the lemma for the morphism $g$. We let the varieties $V_l$ as those defined before Lemma \ref{LemmaGoodPullbackVariety}.  

a) Let $H^p\subset \mathbb{P}^N_K$ be a generic codimension $p$ linear subspace. Then in $A^p(Y)$, $\omega _Y^p$ is represented by $\iota ^*(H^p)$. We can choose such an $H^p$ so that $H^p$ intersects properly $Y$, $g(Z)$ and all irreducible components of $V_l$ and $g(Z)\cap V_l$ for all $l>\dim (X)-\dim (Y)$. By Lemma \ref{LemmaGoodPullbackVariety}, the pullback $g^*(\iota ^*(H^p))=g^{-1}(\iota ^*(H^p))$ is well-defined as a subvariety of $\Gamma$. Moreover, the dimension of $g^{-1}(\iota ^*(H^p))\cap Z$ is less than the dimension of $g^{-1}(\iota ^*(H^p))$. In particular, $g^*(\iota ^*(H^p))$ is effective and has no component on $Z$. 

b) By Roberts' version of Chow's moving lemma, $W$ is rationally equivalent to $\iota^*(\alpha _1)-\iota^*(\alpha _2)\pm W_e$, where $\alpha _1,\alpha _2\subset \mathbb{P}^N_K$  and $W_e\subset Y$ are subvarieties of codimension $p$, and they intersect properly $Y$ and all irreducible components of $V_l$ for all $l>\dim (X)-\dim (Y)$. Moreover, $\deg (\alpha _1),\deg (\alpha _2),\deg (W_e)\leq C\deg (W)$, for some $C>0$ independent of $W$. Moreover, we can find a codimension $p$ variety $\alpha\subset \mathbb{P}^N_K$ so that $\alpha$ intersects properly with $Y$ and all $V_l$, $\iota ^*(\alpha ) -W_e$ is effective, and $\deg (\alpha )\leq C\deg (W_e)$. Note that in $A^p(Y)$ we have $\iota ^*(\alpha _1)\sim \deg (\alpha _1)\omega _Y^p$, $\iota ^*(\alpha _2)\sim \deg (\alpha _2)\omega _Y^p$ and $\iota ^*(\alpha )\sim \deg (\alpha )\omega _Y^p$. Note also that $0\leq g^*(W_e)\leq g^*(\iota ^*(\alpha ))$. Therefore, in $A^p(\Gamma )$
\begin{eqnarray*}
g^*(W)\sim \deg (\alpha _1)g^*(\omega _Y^p)-\deg (\alpha _2)g^*(\omega _Y^p)\pm g^*(W_e),    
\end{eqnarray*}
where each of the three terms on the RHS is effective and $\leq C\deg (W)g^*(\omega _Y ^p)$ for some $C>0$ independent of $W$, $X$ and $f$.
\end{proof}

Finally, we estimate the degree of a strict transform of a cycle by a correspondence. Again, it is sufficient to consider the case $\Gamma _f$ is irreducible. Let $\Gamma \rightarrow \Gamma _f$ be an alteration with the induced morphisms $p:\Gamma \rightarrow X$ and $g:\Gamma \rightarrow Y$.  Define $$g_0=g|_{\Gamma -g^{-1}(\mathcal{C}_g)}: \Gamma -g^{-1}(\mathcal{C}_g)\rightarrow Y-\mathcal{C}_g.$$ Then $g_0$ is a proper morphism, and for any $y\in Y-\mathcal{C}_g$, $g_0^{-1}(y)$ has the correct dimension $\dim (X)-\dim (Y)$. Let $W\subset Y$ be a codimension $p$ subvariety. The inverse image $g_0^{-1}(W)=g^{-1}(W)\cap (\Gamma -g^{-1}(\mathcal{C}_g))\subset \Gamma -g^{-1}(\mathcal{C}_g)$ is a closed subvariety of codimension $p$ of $\Gamma -g^{-1}(\mathcal{C}_g)$, hence its closure $cl(g_0^{-1}(W))\subset \Gamma$ is a subvariety of codimension $p$, and we define $f^{o}(W)=p_*cl(g_0^{-1}(W))$. Note that a strict transform depends on the choice of the alteration $\Gamma$ of the graph $\Gamma _f$. (We can also define a strict transform more intrinsically using the graph $\Gamma _f$ directly, as in \cite{dinh-sibony3}.)

\begin{lemma}
Let $W\subset Y$ be a codimension $p$ subvariety. Then $f^o(W)$ is an effective cycle, and in $A^p(X)$
\begin{eqnarray*} 
f^{o}(W)\leq C\deg (W)f^*(\omega _Y^p),
\end{eqnarray*}
where $C>0$ is a constant independent of the the variety $W$, the manifold $X$  and the correspondence $f$. 
\label{LemmaDegreeOfStrictTransform}\end{lemma}
\begin{proof} Recall again that here we can assume that $\Gamma _f$ is irreducible. That $f^0(W)$ is an effective cycle follows from the definition. It suffices to prove the lemma for the morphism $g:\Gamma \rightarrow Y$. By Roberts' version of Chow's moving lemma, we can decompose $W$ as follows
\begin{eqnarray*}
W=\sum _{i=1}^e(-1)^{i-1}\iota ^*(C_{L_i}(W_{i-1}))+(-1)^eW_e,
\end{eqnarray*} 
where the variety $W_e$ intersects properly all irreducible components of $V_l$ for all $l>\dim (X)-\dim (Y)$, and $C_{i}(W_{i-1})\subset \mathbb{P}^N_K$ are subvarieties of codimension $p$ intersecting $Y$ properly (but may not intersect properly the irreducible components of $V_l$). Moreover, we have the following bound on the degrees
\begin{equation}
\deg (W_e),\deg (C_{L_i}(W_{i-1}))\leq C\deg (W),
\label{Equation0}\end{equation}
for all $i$, where $C>0$ is independent of $W$, $X$ and $f$. 

By the definition of $g^0$ we have
\begin{equation}
g^o(W)= \sum _{i=1}^e(-1)^{i-1}g^o(\iota ^*(C_{L_i}(W_{i-1})))+(-1)^eg^o(W_e).
\label{Equation1}\end{equation}

Note that $e\leq \dim (Y)$. We now estimate each term on the RHS of (\ref{Equation1}). Let $S\subset \mathbb{P}^N_K$ be a subvariety of codimension $p$ intersecting $Y$ properly (but may not intersecting properly the components of $V_l$). We first show that for any such $S$
\begin{equation}
g^o(\iota ^*(S))\leq \deg (S)g^*(\omega _Y^p),
\label{Equation2}\end{equation}
in $A^p(\Gamma )$. 

We can find a curve of automorphisms $\tau (t)\in Aut(\mathbb{P}^N_K)$ for $t\in \mathbb{P}^1_K$ such that for a dense Zariski open dense subset  $U\subset \mathbb{P}^1$, $\tau (t)S$ intersects properly $Y$ and all the irreducible components of $V_l$ (for $l>\dim (X)-\dim (Y)$)  for all $t\in U$. Let $\mathcal{S}\subset Y\times \mathbb{P}^1$ be the corresponding variety, hence for $t\in U\subset \mathbb{P}^1$, $\mathcal{S}_t=\iota ^*(\tau (t)S)\subset Y$. Since $S$ intersects $Y$ properly, we have $\mathcal{S}_0=\iota ^*(S)$. By the choice of $\mathcal{S}$, for any $t\in U$ the pullback $g^*(\mathcal{S}_t)$ is well-defined as a subvariety of $\Gamma$. 

We consider the induced map $G:\Gamma \times \mathbb{P}^1\rightarrow Y\times \mathbb{P}^1$ given by the formula 
\begin{eqnarray*}
G(z,t)=(g(z),t).
\end{eqnarray*}
We define by $G_0$ the restriction map $G_0:\Gamma \times U\rightarrow Y\times U$. By the choice of the variety $\mathcal{S}$, the inverse image $$G_0^{-1}(\mathcal{S})=G^{-1}(\mathcal{S})\cap  (\Gamma \times U)\subset \Gamma \times U$$ is a closed subvariety of codimension $p$, hence its closure $G^o(\mathcal{S})\subset \Gamma \times \mathbb{P}^1$ is a subvariety of codimension $p$. Moreover, for all $t\in U$ we have
\begin{eqnarray*}
G^o(\mathcal{S})_t=g^*(\mathcal{S}_t).
\end{eqnarray*}
Since the map $g_0:\Gamma -g^{-1}(\mathcal{C}_g)\rightarrow Y-\mathcal{C}_g$ has all fibers of the correct dimension $\dim (X)-\dim (Y)$, it follows that $$G^o(\mathcal{S})_0\cap (\Gamma -g^{-1}(\mathcal{C}_g))=g_0^{-1}(\iota ^*(S)).$$ 
In fact, let $G_1$ be the restriction of $G$ to $(\Gamma -\mathcal{C}_g)\times \mathbb{P}^1$. Then $$G_1^{-1}(\mathcal{S} )=G^{-1}(\mathcal{S})\cap [(\Gamma -\mathcal{C}_g)\times \mathbb{P}^1]\subset (\Gamma -\mathcal{C}_g)\times \mathbb{P}^1$$
is a closed subvariety of codimension $p$. Hence its closure, denoted by $\widetilde{G}^o(\mathcal{S})\subset \Gamma \times \mathbb{P}^1$ is a subvariety of codimension $p$. For $t\in U$, we have $\widetilde{G}^o(\mathcal{S})_t=g^*(\mathcal{S}_t)=G^o(\mathcal{S}_t)$, because on the one hand $\widetilde{G}^o(\mathcal{S})_t\subset G^{-1}(\mathcal{S})_t=g^*(\mathcal{S}_t)$, and on the other hand $g^*(\mathcal{S}_t)$ has no component on $g^{-1}(\mathcal{C}_g)$ and $\widetilde{G}^o(\mathcal{S})_t\cap (\Gamma -g^{-1}(\mathcal{C}_g))=g_0^{-1}(\mathcal{S}_t)$. Therefore $\widetilde{G}^o(\mathcal{S})=G^o(\mathcal{S})$ as varieties on $\Gamma \times \mathbb{P}^1$. In particular  
\begin{eqnarray*}
G^o(\mathcal{S})_0\cap (\Gamma -g^{-1}(\mathcal{C}_g))=\widetilde{G}^o(\mathcal{S})_0\cap (\Gamma -g^{-1}(\mathcal{C}_g))=g_0^{-1}(\iota ^*(S)),
\end{eqnarray*}
as claimed.

Hence 
\begin{eqnarray*}
g^o(\iota ^*(S))\leq G^o(\mathcal{S})_0
\end{eqnarray*}
as varieties on $\Gamma$. Since $G^o(\mathcal{S})_0$ is rationally equivalent to $G^o(\mathcal{S})_t$ for any $t$ in $U$, it follows that for all such $t$ we have
\begin{eqnarray*}
g^o(\iota ^*(S))\leq G^o(\mathcal{S})_t=g^*(\mathcal{S}_t)=\deg (S)g^*(\omega _Y^p),
\end{eqnarray*}
in $A^p(\Gamma )$. Hence (\ref{Equation2}) is proven. 

Now we continue the proof of the lemma. By (\ref{Equation2}) and the bound on degrees (\ref{Equation0}), for all $i=1,\ldots ,e$
\begin{eqnarray*}
g^o(\iota ^*(C_{L_i}(W_{i-1})))\leq C\deg (W)g^*(\omega _Y^p),
\end{eqnarray*}
in $A^p(\Gamma )$ where $C>0$ is independent of $W$, $X$ and $f$. 

It remains to estimate $g^o(W_e)$. By the choice of $W_e$, the pullback $g^*(W_e)$ is well-defined as a subvariety of $\Gamma$, hence by b) of Lemma \ref{LemmaDegreeOfPullback} and the bound on degrees (\ref{Equation0}) we have
\begin{eqnarray*} 
g^o(W_e)\leq g^*(W_e)\leq C\deg (W)g^*(\omega _Y^p),
\end{eqnarray*}
in $A^p(\Gamma )$, where $C>0$ is independent of $W$, $X$ and $f$. Thus the proof of the lemma is completed. 
\end{proof} 

\begin{lemma}
Let $X,Y,Z$ be irreducible projective manifolds, and let $f:Y\rightarrow X,~g:Z\rightarrow Y$ be dominant correspondences. We fix an embedding $Y\subset \mathbb{P}^M_K$ and let $\omega _Y$ be the pullback to $Y$ of a generic hyperplane in $\mathbb{P}^M_K$. Then in $A^p(Z)$ 
\begin{eqnarray*} 
(f\circ g)^*(\omega _X^p) \leq C\deg (f^*(\omega _X^p))g^*(\omega _Y^p),
\end{eqnarray*}
where $C>0$ is independent of $f$ and $g$.
\label{LemmaDegreeOfCompositionMaps}\end{lemma}
\begin{proof} We can find proper subvarieties $V_X\subset X,V_Y\subset Y$ and $V_Z\subset Z$ so that the correspondences $f_0:Y-V_Y\rightarrow X-V_X$ and $g_0:Z-V_Z\rightarrow Y-V_Y$ have all fibers of the correct dimensions, by using Theorem \ref{TheoremFiber}. Then it follows that the composition $Z-V_Z\rightarrow X-V_X$ has all fibers of the correct dimension.   

Define by $(f\circ g)_0$ the restriction of $f\circ g$ to $Z-V_Z$. Then $(f\circ g)_0=f_0\circ g_0:Z-V_Z\rightarrow X-V_X$ and has all fibers of the correct dimension. We define the strict transforms $f^0$, $g^0$ and $(f\circ g)^0$ using these restriction correspondences $f_0,g_0$ and $(f\circ g)_0$. 

By Lemma \ref{LemmaDegreeOfPullback} a), we can find a linear subspace $H^p\subset \mathbb{P}^N_K$ so that $H^p$ intersects $X$ properly, $(f\circ g)^*(\iota ^*(H^p))$ is well-defined as a variety and has no component on $V_Z$, and $f^*(\iota ^*(H^p))$ is well-defined as a variety. Then
\begin{eqnarray*}
(f\circ g)^*(\iota ^*(H^p))=(f\circ g)^o(\iota ^*(H^p))
\end{eqnarray*} 
is the closure of $$(f\circ g)_0^{-1}(\iota ^*(H^p))=(f_0\circ g_0)^{-1}(\iota ^*(H^p))=(g_0)^{-1}f_0^{-1}(\iota ^*(H^p)).$$
Therefore $$(f\circ g)^*(\iota ^*(H^p))=g^of^o(\iota ^*(H^p))\leq g^of^*(\iota ^*(H^p)).$$
as subvarieties of $Z$. By Lemma \ref{LemmaDegreeOfStrictTransform}, we have the desired result. 
\end{proof}

\section{Proofs of the main results}
  
 In this section we prove Theorems \ref{TheoremRelativeDynamicalDegrees}, \ref{TheoremWeakerProductFormula}, \ref{TheoremFurtherProperty} and \ref{TheoremTriangleInequality}. The proof of Theorem  \ref{TheoremRelativeDynamicalDegrees} occurs in two first subsections, and we will proceed as follows.
 
 - First, we will prove part 1) of Theorem \ref{TheoremRelativeDynamicalDegrees}, that is to define relative dynamical degrees in the regular case. 
 
 - Then we will use de Jong's alterations and part 1) to define relative dynamical degrees for the general case, and prove other parts of Theorem \ref{TheoremRelativeDynamicalDegrees}.  
 
  From now on, we assume that both $X$ and $Y$ are embedded in the same projective space $\mathbb{P}^N$. Let $H_N\subset \mathbb{P}^N$ be a generic hyperplane. We denote by $\omega _X$ the restriction of $H_N$ to $X$ and $\omega _Y$ the restriction of $H_N$ to $Y$. 
  
 \subsection{The case $X$ is smooth and $\pi$ is a regular morphism}
Because $Y$ has dimension $l$, a generic complete intersection of hyperplane sections $\omega _Y^l$ is a union of $\deg (Y)$ points. The preimages $\pi ^{-1}(\omega _Y^l)\subset X$ will play an important role in the definition of relative dynamical degrees. The advantage of the choice $\pi ^{-1}(\omega _Y^l)$ lies in that these are rationally equivalent and that $\pi ^{-1}(\omega _Y)=(\pi \circ \iota _Y)^*(H_N)$ is a nef class, since $\iota _Y\circ \pi :X\rightarrow \mathbb{P}^N$ is a regular morphism. 

\begin{lemma}
Let $\omega _{X,1}^p$ and $\omega _{X,2}^p$ be two generic complete intersections of hyperplane sections of $X$, and $\omega _{Y,1}^l$ and $\omega _{Y,2}^l$ two generic complete intersections of  hyperplane sections of $Y$. Then the intersections $f^*(\omega _{X,1}^p)\cap \pi ^{-1}(\omega _{Y,1}^l)$ and $f^*(\omega _{X,2}^p)\cap \pi ^{-1}(\omega _{Y,2}^l)$ are well-defined and have the same class in $N^*(X)$.
\label{LemmaRDD2}\end{lemma}
\begin{proof}
Let $\omega _{X,1}^p$ be a fixed complete intersection of hyperplanes sections of $X$ so that $f^{-1}(\omega _{X,1}^p)$ has the correct dimension. This is possible by Lemma \ref{LemmaDegreeOfPullback}.

Let $Z\subset X$ be any proper subvariety of dimension $<\dim (X)$. We show that for a generic choice of $\omega _{Y}$ then $\pi ^{-1}(\omega _{Y})$ intersects $Z$ properly. Otherwise, since $\pi ^{-1}(\omega _{Y})$ is a divisor on $X$ and since $Z$ has only a finite number of irreducible components, there will be an irreducible component  of $Z$ belonging to all $\pi ^{-1}(\omega _{Y})$ - when we vary $\omega _Y$ - which is impossible. (Here the assumption that $\pi$ is regular is used, in that if $A,B$ are two subsets of $Y$ then $\pi ^{-1}(A)\cap \pi ^{-1}(B)=\pi ^{-1}(A\cap B)$.) Using this, starting from $Z_0=f^{-1}(\omega _{X,1}^p)$, there is a generic $\omega _{Y,1}$ such that the wedge intersection $Z_1=Z_0\wedge \pi ^{-1}(\omega _{Y,1})$ is well-defined. Keeping doing this $l$ times we obtain that for a generic choice of hyperplane sections $\omega _{Y,1}^l$, the wedge intersection $f^{-1}(\omega _{X,1}^p)\wedge  \pi ^{-1}(\omega _{Y,1}^l)$ is well-defined. The class in $N^*(X)$ is then the product of the two classes $[f^{-1}(\omega _{X,1}^p)]$ and $[\pi ^{-1}(\omega _{Y,1}^l)]$, and hence is independent of the choice of $\omega _{X,1}^p$ and $\omega _{Y,1}^l$.
\end{proof}

By Lemma \ref{LemmaRDD2}, the number 
\begin{eqnarray*}
\deg _p(f|\pi ):=\deg (f^*(\omega _{X}^p)\cap \pi ^{-1}(\omega _{Y}^l))
\end{eqnarray*}
is well-defined and is independent of the choices of generic complete intersections of hyperplane sections $\omega _X^p$ and $\omega _Y^l$. 

Our next objective is to show that the following limit exists
\begin{eqnarray*}
\lambda _p(f|\pi )=\lim _{n\rightarrow\infty}\deg _p(f^n|\pi )^{1/n}.
\end{eqnarray*}
These limits are the relative dynamical degrees we seek to define. To prove the existence of these limits, we will make use of another quantity associated to $f$. 

{\bf Notation.} Let $\mathcal{B}_f\subset X$ be a proper subvariety, called a "bad" set for $f$, outside it the map $f$ has good properties needed for our purpose. For example, $\mathcal{B}_f$ may contain the critical and indeterminate set of $f$, but it may also contains the critical and indeterminate sets of some iterates of $f$ and the pre images of some "bad" set $\mathcal{B}_g\subset Y$. These "bad" sets are chosen specifically for each of the results proven below, so that the behaviour of certain iterates of $f$ is nice outside $\mathcal{B}_f$. For example, in the proof of Theorem \ref{TheoremRDD4} below, we choose $\mathcal{B}_f$ in such a way that the strict transforms $(f^m)_0$, $(f^n)_0$ and $f^{m+n}_0$ are well-defined and coincide outside $\mathcal{B}_f$. Since it will be quite cumbersome to list $\mathcal{B}_f$ in each case, we will not do so in the sequel.   

Let $V$ be a subvariety of $X$, each of whose irreducible components does not belong to the "bad" set $\mathcal{B}_f$.  Then from the proof of Lemma \ref{LemmaRDD2}, the strict pushforward $f_0(\omega _X^j\cap (V-\mathcal{B}_f))$ is well-defined for a generic complete intersection of  hyperplane sections $\omega _X^j$. Moreover, the class in $N^*(X)$ of the closure of $f_0(\omega _X^j\cap (V-\mathcal{B}_f))$ is independent of the choice of $\mathcal{B}_f$ and such a generic complete intersection of hyperplane sections $\omega _X^j$. We choose, in particular $V=\pi ^{-1}(\omega _Y^l)$ where $\omega _Y^l$ is a generic complete intersection of hyperplane sections of $Y$. Then the class in $N^*(X)$  of the closure of $f_0(\omega _X^j\cap (\pi ^{-1}(\omega _Y^l)-\mathcal{B}_f))$  is independent of the choice of such $\mathcal{B}_f$, $\omega _X^j$ and $\omega _Y^l$. For $j=k-l-p$, we denote the degree of the closure of $f_0(\omega _X^j\cap (\pi ^{-1}(\omega _Y^l)-\mathcal{B}_f))$ by $\deg '_p(f|\pi )$.
\begin{lemma}
For all $0\leq p\leq k-l$, we have: $\deg _p(f|\pi )=\deg '_p(f|\pi )$.
\label{LemmaRDD3}\end{lemma}
\begin{proof}
We choose a generic complete intersection of hyperplane sections $\omega _X^{k-l-p}$ of $X$. By definition
\begin{eqnarray*}
\deg _p(f|\pi )=\omega _X^{k-l-p}.Z,
\end{eqnarray*}
where $Z$ is the closure of $f^*(\omega _X^p).(\pi ^{-1}(\omega _Y^l)-\mathcal{B}_f)$. Since $\omega _X^{k-l-p}$ is a generic complete intersection of hyperplane sections, we have $\omega _X^{k-l-p}\cap Z'=\emptyset$, where $Z'=Z-f^*(\omega _X^p).(\pi ^{-1}(\omega _Y^l)-\mathcal{B}_f)$. Moreover, we can choose so that $\omega _X^{k-l-p}.f^*(\omega _X^p).(\pi ^{-1}(\omega _Y^l)-\mathcal{B}_f)$, which is a $0$-cycle, does not intersect any priori given subvariety of $X$. 

Let $\Gamma _f$ be the graph  of $f$, and let $\pi _1,\pi _2$ be the two projections $X\times X\rightarrow X$. We consider the following $0$-cycle on $X\times X$: $\pi _1^*(\omega _X^{k-l-p}.(\pi ^{-1}(\omega _Y^l)-\mathcal{B}_f)). \pi _2^*(\omega _X^p).\Gamma _f$. Denote by $\alpha$ this zero cycle. The degree of $(\pi _1)_*(\alpha )$ is $\deg _p(f|\pi )$ and the degree of $(\pi _2)_*(\alpha )$ is $\deg '_p(f|\pi )$. Then, the lemma follows from the fact that the degrees of the push forwards by $\pi _1$ and $\pi _2$ of a $0$-cycle on $X\times X$ are the same. 
 \end{proof}

Now we define relative dynamical degrees. By the use of Theorem \ref{TheoremFiber}, we note that the Sard-Bertini's type theorem for characteristic zero is not needed in the algebraic case (compare with the proof of Proposition 3.3. in \cite{dinh-nguyen}).
\begin{theorem}
For any $0\leq p\leq k-l$, the following limit exists
\begin{eqnarray*}
\lambda _p(f|\pi ):=\lim _{n\rightarrow\infty}\deg '_p(f^n|\pi )^{1/n}=\lim _{n\rightarrow \infty}\deg _p(f^n|\pi )^{1/n}.
\end{eqnarray*}
We call $\lambda _p(f|\pi )$ the $p$-th relative dynamical degree of $f$ with respect to $\pi$.
\label{TheoremRDD4}\end{theorem}
\begin{proof}
By Lemma \ref{LemmaRDD3}, we only need to show that the first limit exists. To this end, it suffices to show that there is a constant $C>0$ independent of $m,n$ and $f$ such that
\begin{eqnarray*}
\deg '_p(f^{n+m}|\pi )\leq C\deg '_p(f^n|\pi ).\deg '_p(f^m|\pi ).
\end{eqnarray*}

We define $R=$ the closure of $f^n_0(\omega _X^{k-l-p}\wedge (\pi ^{-1}(\omega _Y^l)-\mathcal{B}))$, where $\mathcal{B}$ is a proper subvariety of $X$, depending on $f^n$, $f^m$ and $f^{n+m}$, so that the strict transform $f^m_0(R)$ is well-defined and equals the closure of $f^{n+m}_0(\omega _X^{k-l-p}\wedge (\pi ^{-1}(\omega _Y^l)-\mathcal{B}))$. By choosing $\omega _X^{k-l-p}$ and $\omega _Y^l$ generically, we can assume that $R$ is a subvariety of codimension $k-l-p$ of $\pi ^{-1}(g_*^n(\omega _Y^l))$ and has proper intersections with the "bad" set $\mathcal{B}$, using the property $\pi \circ f=g\circ \pi$. We note that by definition the degree of $R$ is $\deg '_p(f^n|\pi )$. 

We observe that in $N^*(X)$:
\begin{equation}
R\leq C\deg (R)\omega _X^{k-l-p}\wedge \pi ^{-1}(g_*^n(\omega _Y^l)),
\label{EquationRDD1}\end{equation}
here $C>0$ is a positive constant independent of $R$, $m,n$ and the maps $\pi$, $f$ and $g$. In fact, consider the embedding $\pi ^{-1}(g_*^n(\omega _Y^l))\subset \mathbb{P}^N$ induced from the embedding $X\subset \mathbb{P}^N$. Regardless whether $\pi ^{-1}(g_*^n(\omega _Y^l))$ is smooth or not, it has the correct dimension $k-l$. We then can use Roberts' version of Chow's moving lemma, to have that if $L\subset \mathbb{P}^N$ is any linear subspace of dimension $n-(k-l)-1$ and $C_L(R)$ is the cone over $R$ with vertex $L$, then $C_L(R)\cap \pi ^{-1}(g_*^n(\omega _Y^l))$ has the correct dimension $\dim (R)$. In particular, since $R$ is one component of $C_L(R)\cap \pi ^{-1}(g_*^n(\omega _Y^l))$ and the degree of $C_L(R)$ is bounded from above by $C\deg (R)$, we obtain Equation (\ref{EquationRDD1}). 

Using the same arguments as in the proof of Lemma \ref{LemmaDegreeOfStrictTransform}, we can show moreover that  in $N^*(X)$:
\begin{equation}
f^m_0(R)\leq C\deg (R)f^m_0(\omega _X^{k-l-p}\wedge \pi ^{-1}(g_*^n(\omega _Y^l))).
\label{EquationRDD2}\end{equation}
This completes the proof of Theorem \ref{TheoremRDD4}, by noting that $\pi ^{-1}(g_*^n(\omega _Y^l))$ is contained in at most $\deg (Y)$ fibers of the form $\pi ^{-1}(\omega _Y^l)$.
\end{proof}
Hence, part 1) of Theorem \ref{TheoremRelativeDynamicalDegrees} is proven, except we did not show that the limit is independent of the choice of the embeddings $\iota _X$ and $\iota _Y$. But this is clear since $\omega _X$ is ample and $\omega _Y^l$ has the same class in $N^*(Y)$ as a positive multiple of the class of a point, and relative dynamical degrees can be computed on numerical classes. That these numbers are $\geq 1$ follows easily from the fact that they are limits of the $n$-th roots of degrees of non-trivial subvarieties of $X$, and the latter are positive integers.

\subsection{The general case} In this subsection we prove Theorem \ref{TheoremRelativeDynamicalDegrees} in the general case. Hence, we do not assume that $X$ or $Y$ is smooth or $\pi$ is regular. We only require that $X$ and $Y$ are irreducible. We need to make use of the following special case of Theorem \ref{TheoremFurtherProperty}.
\begin{lemma}
Let $X,Z$ be irreducible smooth projective varieties and $Y$ an (not necessarily smooth) irreducible projective variety. Assume that $\dim (X)=\dim (Z)$, and there are two semi-conjugacies $\varphi :(Z,h)\rightarrow (X,f)$ and $\pi :(X,f)\rightarrow (Y,g)$, where both $\varphi $ and $\pi $ are dominant regular morphisms. Then the relative dynamical degrees (as defined in the previous subsection) of the semi-conjugacies  $\pi :(X,f)\rightarrow (Y,g)$ and $\pi \circ \varphi :(Z,h)\rightarrow (Y,g)$ are the same.
\label{LemmaSemiconjugacyInvariant}\end{lemma}
\begin{proof}
Let $U\subset X$ be a non-empty Zariski open set such that all fibers $\varphi ^{-1}(x)$ are finite. Define $U'=\varphi ^{-1}(U)$. Let $(f^n)_0$ and $(h^n)_0$ be the strict pushforward with respect to $U$ and $\varphi ^{-1}(U)$. Choose $\omega _X^p$, $\omega _Z^p$ and $\omega _Y^l$ to be generic hyperplane sections. By duality, we have 
\begin{eqnarray*}
\lambda _p(h|\pi \circ \varphi )=\lim _{n\rightarrow \infty}<\omega _Z^p,(h^n)_*(\varphi ^*\pi ^*(\omega _Y^l)\wedge \omega _Z^{k-l-p})>^{1/n}.
\end{eqnarray*}
Fix $n\in \mathbb{N}$. We now estimate the number $<\omega _Z^p,(h^n)_*(\varphi ^*\pi ^*(\omega _Y^l)\wedge \omega _Z^{k-l-p})>$. From the previous results, we see that by choosing $U$ appropriately, we may assume that $(h^n)|_{U'}(\varphi ^*\pi ^*(\omega _Y^l)\wedge \omega _Z^{k-l-p})$ is well-defined as a variety, of the same class in $N^*(Z)$ as $(h^n)_*(\varphi ^*\pi ^*(\omega _Y^l)\wedge \omega _Z^{k-l-p})$, the intersection $\omega _Z^p|_{U'}\wedge (h^n)|_{U'}(\varphi ^*\pi ^*(\omega _Y^l)\wedge \omega _Z^{k-l-p})$ is well-defined and is the same as $<\omega _Z^p,(h^n)_*(\varphi ^*\pi ^*(\omega _Y^l)\wedge \omega _Z^{k-l-p})>$, and 
\begin{eqnarray*}
(\varphi _{U'})_*(h^n)|_{U'}(\varphi ^*\pi ^*(\omega _Y^l)\wedge \omega _Z^{k-l-p})=(f^n)|_U(\varphi _{U'})_*\varphi ^*\pi ^*(\omega _Y^l)\wedge \omega _Z^{k-l-p}.
\end{eqnarray*}
Define $\varphi '=\varphi |_{U'}$. Let $a$ be the degree of $\varphi$. Since $\omega _Z^p|_{U'}\wedge (h^n)|_{U'}(\varphi ^*\pi ^*(\omega _Y^l)\wedge \omega _Z^{k-l-p})$ is a finite number of points, and $\varphi :U'\rightarrow U$ is a finite map, it follows that the wedge intersection $\varphi '_*(\omega _Z^p)|_{U'}\wedge \varphi '_*[(h^n)|_{U'}(\varphi ^*\pi ^*(\omega _Y^l)\wedge \omega _Z^{k-l-p})]$ is well-defined. Moreover, since the degree of $\varphi $ is $a$, for any point $x\in U$ there are at most $a$ points in $\varphi ^{-1}(x)$. Thus, we obtain
\begin{eqnarray*}
<\omega _Z^p,(h^n)_*(\varphi ^*\pi ^*(\omega _Y^l)\wedge \omega _Z^{k-l-p})>&=&\omega _Z^p|_{U'}\wedge (h^n)|_{U'}(\varphi ^*\pi ^*(\omega _Y^l)\wedge \omega _Z^{k-l-p})\\
&\leq& a\varphi '_*(\omega _Z^p)\wedge \varphi '_*[(h^n)|_{U'}(\varphi ^*\pi ^*(\omega _Y^l)\wedge \omega _Z^{k-l-p})]. 
\end{eqnarray*}
Since $\varphi \circ h=f\circ \varphi$, we can assume also that $\varphi '_*[(h^n)|_{U'}(\varphi ^*\pi ^*(\omega _Y^l)\wedge \omega _Z^{k-l-p})]=(f^n)|_{U}\varphi '_*[\varphi ^*\pi ^*(\omega _Y^l)\wedge \omega _Z^{k-l-p}]$. We then obtain
\begin{eqnarray*}
<\omega _Z^p,(h^n)_*(\varphi ^*\pi ^*(\omega _Y^l)\wedge \omega _Z^{k-l-p})>\leq a \varphi '_*(\omega _Z^p)\wedge (f^n)|_{U}[\pi ^*(\omega _Y^l)\wedge \varphi '_*(\omega _Z^{k-l-p})]. 
\end{eqnarray*}  
We can choose $U$ such that $(f^n)|_{U}[\pi ^*(\omega _Y^l)\wedge \varphi '_*(\omega _Z^{k-l-p})=(f^n)_*[\pi ^*(\omega _Y^l)\wedge \varphi '_*(\omega _Z^{k-l-p}]$. Using the Main Lemma in \cite{roberts}, as in the proof of Lemma \ref{LemmaChowMovingForProduct}, we can find a cone $C_L(\varphi '_*(\omega _Z^p))$ which intersects $(f^n)|_{U}[\pi ^*(\omega _Y^l)\wedge \varphi '_*(\omega _Z^{k-l-p})$ properly on $U$, and by choosing a $1$-family of automorphisms $g(t)$ of $\mathbb{P^N}$ we have that $g(t)C_L(\varphi '_*(\omega _Z^p))|_{X}$ intersects $(f^n)|_{U}[\pi ^*(\omega _Y^l)\wedge \varphi '_*(\omega _Z^{k-l-p})$ properly, for generic $t\in \mathbb{P}^1$. From this we obtain, where from the second line the computations are in $N^*(X)$ and we use that the class of $g(t)C_L(\varphi '_*(\omega _Z^p)|_X$ is a constant multiple of $\omega _X^p$, where the constant depends only on $\varphi$ and $\omega _Z$. 
\begin{eqnarray*}
&&\varphi '_*(\omega _Z^p)\wedge (f^n)|_{U}[\pi ^*(\omega _Y^l)\wedge \varphi '_*(\omega _Z^{k-l-p})]\\
&\leq &g(t)C_L(\varphi '_*(\omega _Z^p))|_{X}.(f^n)_*[\pi ^*(\omega _Y^l)\wedge \varphi '_*(\omega _Z^{k-l-p}]\\ 
&=&<[g(t)C_L(\varphi '_*(\omega _Z^p))|_{X}],(f^n)_*[\pi ^*(\omega _Y^l)\wedge \varphi '_*(\omega _Z^{k-l-p}]>\\
&=&<\omega _X^p,(f^n)_*[\pi ^*(\omega _Y^l)\wedge \varphi '_*(\omega _Z^{k-l-p}]>\\
&=&<[(f^n)^*(\omega _X^p)\wedge \pi ^*(\omega _Y^l)],[\varphi _*\omega _Z^{k-l-p}]>\\
&\leq&A<[(f^n)^*(\omega _X^p)\wedge \pi ^*(\omega _Y^l)],[\omega _X^{k-l-p}]>\\&&
\times <[\varphi _*\omega _Z^{k-l-p}],[\omega _X^{k+l}]>\\
&\leq &A'<[(f^n)^*(\omega _X^p)\wedge \pi ^*(\omega _Y^l)],[\omega _X^{k-l-p}]>\\
&=&A'\deg _p(f^n|\pi ).
\end{eqnarray*}
In the fourth line, we have used the proof of Lemma \ref{LemmaRDD2} to represent the class $[(f^n)^*(\omega _X^p)].[\pi ^*(\omega _Y^l]$ by the class of a variety $(f^n)^*(\omega _X^p)\wedge \pi ^*(\omega _Y^l)$. The fifth line follows from the fourth line by applying Lemma \ref{LemmaDegreeOfIntersections}. 

From the above computation, we obtain $\lambda _p(h|\pi \circ \varphi )\leq \lambda _p(f|\pi )$. The reverse inequality can be proven similarly (and easier). 
\end{proof}

Now we proceed to defining relative dynamical degrees in this general case. 
Let $\tau :X'\rightarrow X$ be an alteration, where $\tau$ is generically finite and so that $\pi \circ \tau$ is a regular morphism. Let $\tau ^*f$ be the pullback of the correspondence $f$ by $\tau$, as defined in Section 3. We then define, using the regular case considered in the previous subsection
\begin{eqnarray*}
\lambda _p(f|\pi ):=\frac{1}{\deg (\tau )}\lambda _p(\tau ^*f|\tau \circ \pi ). 
\end{eqnarray*}
To show that this is well-defined, we need to show the following. Let $\tau _1:X_1\rightarrow X$ and $\tau _2:X_2\rightarrow X$ be two alterations, so that $\pi \circ \tau _1$ and $\pi \circ \tau _2$ are regular, then
\begin{eqnarray*}
\frac{1}{\deg (\tau _1)}\lambda _p(\tau _1^*f|\tau _1\circ \pi )=\frac{1}{\deg (\tau _2)}\lambda _p(\tau _2^*f|\tau _2\circ \pi ).
\end{eqnarray*}
To this end, we first observe that there is another alteration $\tau _3:X_3\rightarrow X$ together with two generically finite morphisms $\varphi _1:X_3\rightarrow X_1$ and $\varphi _2:X_3\rightarrow X_2$ such that $\tau _1\circ \varphi _1=\tau _3=\tau _2\circ \varphi _2$. In fact, let $\Gamma \subset X_1\times X_2$ be an irreducible component of the correct dimension of the set $\{(x_1,x_2)\in X_1\times X_2:~\tau _1(x_1)=\tau _2(x_2)\}$ and which is dominant over $X_1$ and $X_2$. Then any alteration $X_3$ of $\Gamma$ will satisfy the requirement. 

Since $\tau _3=\tau _1\circ \varphi _1=\tau _2\circ \varphi _2$, we have $\varphi _1^*(\tau _1^*f)=\tau _3^*f=\varphi _2^*(\tau _2^*f)$. By Lemma \ref{LemmaSemiconjugacy1}, we have semi-conjugacies $(X_3,\tau _3^*f)\rightarrow (X_1,\deg (\varphi _1)\tau _1^*f)$ and $(X_3,\tau _3^*f)\rightarrow (X_2,\deg (\varphi _2)\tau _2^*f)$. Writing $\pi _i=\pi \circ \tau _i$ for $i=1,2,3$, by Lemma \ref{LemmaSemiconjugacyInvariant}, we have 
\begin{eqnarray*}
\deg (\varphi _1)\lambda _p(\tau _1^*f|\pi _1)=\lambda _p(\tau _3^*f|\pi _3)=\deg (\varphi _2)\lambda _p(\tau _2^*f|\pi _2),
\end{eqnarray*}
which, together with the equality $\deg (\varphi _1)\deg (\tau _1)=\deg (\tau _3)=\deg (\varphi _2)\deg (\tau _2)$, gives us the desired equality
\begin{eqnarray*}
\frac{1}{\deg (\tau _1)}\lambda _p(\tau _1^*f|\pi _1)=\frac{1}{\deg (\tau  _3)}\lambda _p(\tau _3^*f| \pi _3)=\frac{1}{\deg (\tau _2)}\lambda _p(\tau _2^*f|\pi _2). 
\end{eqnarray*}

Thus the relative dynamical degrees in this case are well-defined. Part 2) of Theorem \ref{TheoremRelativeDynamicalDegrees} follows from the following lemma. 
\begin{lemma}
Let be given two semi-conjugacies $\pi _1:(X_1,f_1)\rightarrow (Y,g)$ and $\pi _2:(X_2,f_2)\rightarrow (Y_2,g_2)$, together with two  semi-conjugacies $\varphi :(X_2,f_2)\rightarrow (X_1,f_1)$ and $\psi : (Y_2,g_2)\rightarrow (Y_1,g_1)$ where the two composition maps $\psi \circ \pi _2$ and $\pi _1\circ \varphi$ are the same. Assume that $\psi :Y_2\rightarrow Y_1$ is a generically finite rational map, and moreover that $\varphi$ is birational.  Then $\lambda _p(f_1|\pi _1)=\lambda _p(f_2|\pi _2)$ for all $p=0,\ldots ,\dim (X_1)-\dim (Y_1)$.
\label{LemmaBirationalInvariant}\end{lemma}
\begin{proof}
Since $\varphi :X_2\rightarrow X_1$ is a birational map, it can be easily checked that $f_2=\varphi ^*(f_1)$. Let $\tau _2:Z\rightarrow X_2$ be an alteration (so $Z$ is smooth and $\tau _2$ is a generically finite regular morphism) so that  the composition maps $\tau _1=\varphi \circ \tau _2:Z\rightarrow X_1$, $\pi _2'=\pi _2\circ \tau _2:Z\rightarrow Y_2$ and $\pi _1'=\pi _1\circ  \tau _1:Z\rightarrow Y_1$ are all regular. Let $h:Z\rightarrow Z$ be the pullback of $f_2$ by $\tau _2$, that is $h=\tau _2^*f_2$. Then from $\tau _1=\varphi \circ \tau _2$ and $f_2=\varphi ^*(f_1)$, we have $\tau _1^*(f_1)=\tau _2^*\varphi ^*(f_1)=\tau _2^*(f_2)$. Therefore, from $\deg (\tau _1)=\deg (\tau _2)$ since $\pi$ is birational, by definition we obtain
\begin{eqnarray*}
\lambda _p(f_1|\pi _1)&=&\frac{1}{\deg (\tau _1)}\lambda _p(\tau _1^*(f_1)|\pi _1')\\
&=&\frac{1}{\deg (\tau _2)}\lambda _p(\tau _2^*(f_2)|\pi _2')\\
&=&\lambda _p(f_2|\pi _2).
\end{eqnarray*}     
Here the second equality between $\lambda _p(\tau _1^*(f_1)|\pi _1')$ and $\lambda _p(\tau _2^*(f_2)|\pi _2')$ remains to be explained, since the spaces involved are different. Define $f:=\tau _1^*(f_1):Z\rightarrow Z$, which is also the same as $\tau ^*(f_2)$. From the assumption of the lemma, we have $\pi _1'=\psi \circ \pi _2'$. Therefore, $(\pi _1')^*(\omega _{Y_1}^l)=(\pi _2')^*(\psi ^*(\omega _{Y_1}^l))$. Using part 1) of Theorem \ref{TheoremRelativeDynamicalDegrees}, which says that relative dynamical degrees can be computed using numerical classes, and the fact that in $N^*(Y_2)$ the two classes of $\psi ^*(\omega _{Y_1}^l)$ and $\omega _{Y_2}^l$ are positive multiples of each other, the desired equality follows. 
\end{proof}

Now we proceed to proving the log-concavity of relative dynamical degrees i.e. part 3) of Theorem \ref{TheoremRelativeDynamicalDegrees}. We note that in the case the underlying field is $\mathbb{C}$, our proof here is different from and simpler than that used in Proposition 3.6 of \cite{dinh-nguyen}. 

\begin{lemma} Assume that there is an infinite sequence $n_i\in \mathbb{N}$ for which the graph of $f^{n_i}$ is irreducible for all $i$. Then, the relative dynamical degrees are log-concave, that is $\lambda _{p+1}(f|\pi )\lambda _{p-1}(f|\pi )\leq \lambda _p(f|\pi )^2$ for all $1\leq p\leq k-l-1$. 
\label{LemmaLogConcavityRelativeDynamicalDegrees}\end{lemma}
\begin{proof}
If $X$ is smooth and $\pi$ is regular, then the result follows by applying the Grothendieck-Hodge index theorem to an alternation of $\Gamma _{f^{n_i}}$, using part 1) of Theorem \ref{TheoremRelativeDynamicalDegrees}. In the general case we proceed as follows.

Let $\tau _1:X_1\rightarrow X$ be an alteration of degree $d_1$ so that $X_1$ is smooth and the composition $\pi _1=\pi \circ \tau _1:X_1\rightarrow Y$ is a regular morphism. Define $f_1=\tau _1^*f:X_1\rightarrow X_1$. By definition, we have
\begin{eqnarray*}
\lambda _p(f|\pi )=\frac{1}{d_1}\lambda _p(f_1|\pi _1)
\end{eqnarray*} 
From the proof of Lemma \ref{LemmaRDD2}, it can be easily shown that 
\begin{eqnarray*}
\lambda _p(f_1|\pi _1)=\lim _{n\rightarrow\infty}<(f_1^{n})^*\tau _1^*(\omega _X^p)\wedge \pi _1^*(\omega _Y^l),\tau _1^*(\omega _X^{k-l-p})>^{1/n}. 
\end{eqnarray*}

Now we prove the log-concavity of relative dynamical degrees. To this end, it suffices to show that 
\begin{eqnarray*}
\deg _{p+1}(f|\pi )\deg _{p-1}(f|\pi )\leq \deg _p(f|\pi )^2. 
\end{eqnarray*}
Choose $n_i$ to be a positive integer for which the graph $\Gamma _{f}^{n_i}$ is irreducible. Let $\Gamma\rightarrow \Gamma _{f^{n_i}}$ be an alteration of degree $d$ of the graph of $f^{n_i}$, and let $\pi _1,\pi _2:\Gamma \rightarrow X$ be the two projections. We write $\Gamma _{f_1^{n_i}}=\sum _{q\in I_{n_i}}G_q$ for some finite set $I_{n_i}$, where $G_q$ are irreducible components. For each $q\in I_{n_i}$, there is an alteration $\widetilde{G_q}\rightarrow G_q$ of degree $a_q$ which is equipped with a regular morphism $\alpha _q:\widetilde{G_q}\rightarrow \Gamma$. Denote by $p_{1,q},p_{2,q}:\widetilde{\Gamma _q}\rightarrow X$ the natural projections.  By definition
\begin{eqnarray*}
&&<(f_1^{n})^*\tau _1^*(\omega _X^p)\wedge \pi _1^*(\omega _Y^l),\tau _1^*(\omega _X^{k-l-p})>\\
&=&\sum _{q\in I_{n_i}}\frac{1}{a_q}p_{1,q}^*. \pi _1^*(\omega _Y^l).p_{1,q}^*\tau _1^*(\omega _X^{k-l-p}).p_{2,q}^*\tau _1^*(\omega _X^p).
\end{eqnarray*}
We now show that each summand in the RHS of the above equality is a positive multiple of the following
\begin{eqnarray*}
p_{1}^*. \pi ^*(\omega _Y^l).p_{1}^*(\omega _X^{k-l-p}).p_{2}^*(\omega _X^p).
\end{eqnarray*}
Moreover, the multiples depend only on the degrees of the morphisms $\widetilde{G_q}\rightarrow \Gamma$, and are independent of $p=\{0,\ldots ,k-l\}$.  This claim can be easily seen by chasing various commutative diagrams of surjective morphisms, see also the proof of Lemma \ref{LemmaWedgeIntersectionCompatibility} below. Since the last number is an intersection of nef classes on the irreducible smooth variety $\Gamma$, we can approximate them by $ \mathbb{Q}$-ample classes, and use the usual Grothendieck-Hodge index theorem to prove the log-concavity.
\end{proof}
Hence, part 3) of Theorem \ref{TheoremRelativeDynamicalDegrees} is proven. Now we proceed to proving part 4) of Theorem \ref{TheoremRelativeDynamicalDegrees}, which also finishes the proof of the Theorem, because part 5) is standard. Provided that we can prove results from Section 3 in \cite{dinh-nguyen-truong1} (c.f. also Section 4 in \cite{dinh-nguyen}), part 4) will follow easily. We can assume that $X$ is a smooth projective variety and $\pi :X\rightarrow Y$ is a surjective morphism.  Let us first discuss the idea for the proof of part 4). We note that by arguments as in the proofs of Lemmas \ref{LemmaChowMovingForProduct} and  \ref{LemmaSemiconjugacyInvariant}, the inequality $a\lambda _p(f) \leq \max _{0\leq j\leq l,~0\leq p-j\leq k-l}\lambda _j(g)\lambda _{p-j}(f|\pi )$ follows provided we can prove Lemma 4.3 in \cite{dinh-nguyen}, while the reverse inequality $a\lambda _p(f) \geq \max _{0\leq j\leq l,~0\leq p-j\leq k-l}\lambda _j(g)\lambda _{p-j}(f|\pi )$ follows provided we can prove Lemma 4.5 in \cite{dinh-nguyen}. (Remark that in \cite{dinh-nguyen-truong1} we simplified the treatment of part 4) in \cite{dinh-nguyen} by using an analogue of Lemma \ref{LemmaChowMovingForProduct}, so that we do not need to consider a surjective morphism $\Pi =(\pi ,v):X\rightarrow Y\times \mathbb{P}^{k-l}$. )

Lemma 4.3 in \cite{dinh-nguyen}  (see also \cite{dinh-nguyen-truong1}) was proven, in the case $f$ is a dominant rational map, by using the following (which is in lines 2-3 in the proof of that lemma): $(f^n)^*(\pi ^*(\omega _Y^{p-q})\wedge \omega _X^{q})\leq (f^n)^*(\pi ^*(\omega _Y^{p-q}))\stackrel{o}{\wedge}(f^n)^*(\omega _X^{q})$. In our case, this would amount to proving that for a general correspondence $a^n(f^n)^*(\pi ^*(\omega _Y^{p-q})\wedge \omega _X^{q})\leq (f^n)^*(\pi ^*(\omega _Y^{p-q}))\stackrel{o}{\wedge}(f^n)^*(\omega _X^{q})$. While this holds for $a=1$, there is no reason why it should be true for $a>1$.  

Lemma 4.5 in \cite{dinh-nguyen} (see also \cite{dinh-nguyen-truong1}) was proven, in the case $f$ is a dominant rational map, using  the following (which is in lines 7-9 in the proof of that lemma): $(f^n)^*(\pi ^*(\omega _Y^{p-q})\wedge \omega _X^{q})\geq (f^n)^*(\pi ^*(\omega _Y^{p-q}))\stackrel{o}{\wedge}(f^n)^*(\omega _X^{q})$. In our case, this would amount to proving that for a general correspondence $a^n(f^n)^*(\pi ^*(\omega _Y^{p-q})\wedge \omega _X^{q})\geq (f^n)^*(\pi ^*(\omega _Y^{p-q}))\stackrel{o}{\wedge}(f^n)^*(\omega _X^{q})$. This seems to be false, even if we choose $a$ as large as we like, by considering as simple examples as $f=f_1+f_2$, where $f_1,f_2$ are two dominant rational maps. 

To overcome these difficulties when working with general correspondences, we use the observation that both Lemmas 4.3 and 4.5 in \cite{dinh-nguyen} follow if we can show the following: If $f$ is a dominant rational map (hence so is $g$), then the two numbers $a^n<(f^n)^*(\pi ^*(\omega _Y^{p-q})\wedge \omega _X^{q}),\pi ^*(\omega _Y^{l-p+p_0})\wedge \omega _X^{k-l-p_0}>$ and $(f^n)^*\pi ^{*}(\omega _Y^{p-q})\wedge \pi ^*(\omega _Y^{l-p+p_0})\stackrel{o}{\wedge}(f^n)^*(\omega _X^q)\wedge \omega _X^{k-l-p_0}$ are compatible. Using the arguments in the proof of Lemma \ref{LemmaLogConcavityRelativeDynamicalDegrees} to each irreducible component of $\Gamma _{g^n}$, we are able to achieve this estimate under the assumption that $g$ is a multiple of a rational map.

\begin{lemma}
Let $X,Y$ be irreducible smooth projective varieties of corresponding dimensions $k\geq l$ , and $\pi :(X,f)\rightarrow (Y,g)$ a semi-conjugacy where $\pi :X\rightarrow Y$ is a surjective regular morphism. Assume that $g=ag'$ where $g':Y\rightarrow Y$ is a dominant rational map. Let $\omega _X$ be an ample divisor on $X$ and $\omega _Y$ an ample divisor on $Y$. Then, there are two positive constants $C_1,C_2$  such that for all $n\in \mathbb{N}$ and $p,q,r\in \mathbb{N}$ we have
\begin{eqnarray*}
&&C_1 (f^n)^*\pi ^{*}(\omega _Y^{p})\wedge \pi ^*(\omega _Y^{r})\stackrel{o}{\wedge}(f^n)^*(\omega _X^q)\wedge \omega _X^{k-p-q-r}\\
&\leq& a^n<(f^n)^*(\pi ^*(\omega _Y^p)\wedge \omega _X^q),\pi ^*(\omega _Y^r)\wedge \omega _X^{k-p-q-r}>\\
&\leq& C_2(f^n)^*\pi ^{*}(\omega _Y^{p})\wedge \pi ^*(\omega _Y^{r})\stackrel{o}{\wedge}(f^n)^*(\omega _X^q)\wedge \omega _X^{k-p-q-r}.
\end{eqnarray*}
\label{LemmaWedgeIntersectionCompatibility}\end{lemma} 
\begin{proof}
If $f$ is a dominant rational map, then the arguments in the proofs of Lemmas 4.3 and 4.5 in \cite{dinh-nguyen} applies almost verbatim, pending the simplifications made in \cite{dinh-nguyen-truong1}. In the general case, we argue as follows. 

Fix a number  $n$, and let $G_n\subset Y\times Y$ be the graph of the dominant rational map $(g')^n:Y\rightarrow Y$. Let $\Gamma _{f^{n}}=\sum _{i}\Gamma _i$ where $\Gamma _i$ are irreducible varieties. Let $\Gamma$ be one among the $\Gamma _i$'s. Since $\pi :(X,f)\rightarrow (Y,g)$ is a semi-conjugacy, it follows that the image of $\Gamma$ under the regular morphism $\pi \times \pi :X\times X\rightarrow Y\times Y$ must be  $G_n$. We let $\widetilde{G_n}\rightarrow G_n$ be an alteration of degree say $d$, and $p_1,p_2:\widetilde{G_n}\rightarrow Y$ be the two natural projections. We can choose an alteration $\widetilde{\Gamma}\rightarrow \Gamma$ of degree say $d'$ which is equipped with a regular surjective morphism $\gamma :\widetilde{\Gamma}\rightarrow \widetilde{G_n}$ so that the two compositions $\widetilde{\Gamma}\rightarrow \Gamma \rightarrow G_n$ and $\widetilde{\Gamma}\rightarrow \widetilde{G_n}\rightarrow G_n$ coincide. Let $p_1',p_2':\widetilde{\Gamma}\rightarrow X$ be the two natural projections.   

Then
\begin{eqnarray*}
&&<\Gamma ^*(\pi ^*(\omega _Y^p)\wedge \omega _X^q),\pi ^*(\omega _Y^r)\wedge \omega _X^{k-p-q-r}>\\
&=&\frac{1}{d'}(p_1')^*\pi ^*(\omega _Y^r)\wedge (p_1')^*(\omega _X^{k-p-q-r})\wedge (p_2')^*\pi ^*(\omega _Y^p)\wedge (p_2')^*(\omega _X^q)\\
&=&\frac{1}{d'}\gamma ^*(p_1^*(\omega _Y^r)\wedge p_2^*(\omega _Y^p))\wedge (p_1')^*(\omega _X^{k-p-q-r})\wedge (p_2')^*(\omega _X^q)\\
&\sim&\frac{1}{d'}\frac{1}{d}\gamma ^*p_1^*(p_1)_*(p_1^*(\omega _Y^r)\wedge p_2^*(\omega _Y^p))\wedge (p_1')^*(\omega _X^{k-p-q-r})\wedge (p_2')^*(\omega _X^q)\\
&\sim&(p_1')^*\pi ^*[\frac{1}{d}(p_1)_*p_2^*(\omega _Y^p)\wedge \omega _Y^r]\wedge (p_1')^*(\omega _X^{k-p-q-r})\wedge [\frac{1}{d'}(p_2')^*(\omega _X^q)]\\
&\sim&\pi ^*[G_n^*(\omega _Y^p)\wedge \omega _Y^{r}]\wedge \Gamma ^*(\omega _X^{q})\wedge \omega _X^{k-p-q-r}.
\end{eqnarray*}
Here $\sim$ means that the two numbers involved are positive multiples of each other, where the multiples are constants bounded in terms of $X$ and $Y$ only. We used in the fourth line of the above expression that $$\frac{1}{d}p_1^o(p_1)_*(p_1^*(\omega _Y^r)\wedge p_2^*(\omega _Y^p))=p_1^*(\omega _Y^r)\wedge p_2^*(\omega _Y^p),$$
(here $p_1^o$ on the far left means the strict pullback by $p_1$) which follows from the fact that the projection to the first factor $G_n\rightarrow Y$ is a birational morphism, since $G_n$ is the graph of a (dominant) rational map on $Y$. 

If we take the sum all over $\Gamma$'s, the proof of the lemma is completed, by using Lemma \ref{LemmaChowMovingForProduct} and the fact that $(g^n)^*=a^nG_n^*$. 
\end{proof}

\subsection{Proof of Theorem \ref{TheoremWeakerProductFormula}}

We observe that it is enough to prove Proposition 4.5 in \cite{dinh-nguyen}. Checking the proof of that proposition carefully, we see that it suffices to show the following: for every $a>\lambda _0(g)$  
\begin{eqnarray*}
&&a^n<(f^n)^*(\pi ^*(\omega _Y^j)\wedge \omega _X^{p-j}),\pi ^*(\omega _Y^{l-j})\wedge \omega _X^{k-p-q+j}>\\
&\geq& \pi ^*[(g^n)^*(\omega _Y^{j})\wedge \omega _Y^{l-j}]\stackrel{o}{\wedge}(f^n)^*(\omega _X^{p-j})\wedge \omega _X^{k-l-p+j},
\end{eqnarray*}
provided that $n$ is large enough and $g^n$ is irreducible. 

To this end, we follow the proof of Lemma \ref{LemmaWedgeIntersectionCompatibility}. Let $n$ be an integer for which $g^n$ is irreducible, and write $g^n=a_nG_n$, where $G_n\subset Y\times Y$ is irreducible.  Let $\Gamma _{f^{n}}=\sum _{i}\Gamma _i$ where $\Gamma _i$ are irreducible varieties. Let $\Gamma$ be one among the $\Gamma _i$'s. Since $\pi :(X,f)\rightarrow (Y,g)$ is a semi-conjugacy, it follows that the image of $\Gamma$ under the regular morphism $\pi \times \pi :X\times X\rightarrow Y\times Y$ must be  $G_n$. We let $\widetilde{G_n}\rightarrow G_n$ be an alteration of degree say $d$, and $p_1,p_2:\widetilde{G_n}\rightarrow Y$ be the two natural projections. We can choose an alteration $\widetilde{\Gamma}\rightarrow \Gamma$ of degree say $d'$ which is equipped with a regular surjective morphism $\gamma :\widetilde{\Gamma}\rightarrow \widetilde{G_n}$ so that the two compositions $\widetilde{\Gamma}\rightarrow \Gamma \rightarrow G_n$ and $\widetilde{\Gamma}\rightarrow \widetilde{G_n}\rightarrow G_n$ coincide. Let $p_1',p_2':\widetilde{\Gamma}\rightarrow X$ be the two natural projections. Let $b_n$ be the degree of the morphism $G_n\rightarrow Y$.  

Then
\begin{eqnarray*}
&&<\Gamma ^*(\pi ^*(\omega _Y^j)\wedge \omega _X^{p-j}),\pi ^*(\omega _Y^{l-j})\wedge \omega _X^{k-l-p+j}>\\
&=&\frac{1}{d'}(p_1')^*\pi ^*(\omega _Y^{l-j})\wedge (p_1')^*(\omega _X^{k-l-p+j})\wedge (p_2')^*\pi ^*(\omega _Y^j)\wedge (p_2')^*(\omega _X^{p-j})\\
&=&\frac{1}{d'}\gamma ^*(p_1^*(\omega _Y^{l-j})\wedge p_2^*(\omega _Y^j))\wedge (p_1')^*(\omega _X^{k-l-p+j})\wedge (p_2')^*(\omega _X^{p-j})\\
&=&\frac{1}{d'}\frac{1}{d}\frac{1}{b_n}\gamma ^*p_1^*(p_1)_*(p_1^*(\omega _Y^{l-j})\wedge p_2^*(\omega _Y^j))\wedge (p_1')^*(\omega _X^{k-l-p+j})\wedge (p_2')^*(\omega _X^{p-j})\\
&=&\frac{1}{b_n}(p_1')^*\pi ^*[\frac{1}{d}(p_1)_*p_2^*(\omega _Y^{j})\wedge \omega _Y^{l-j}]\wedge (p_1')^*(\omega _X^{k-l-p+j})\wedge [\frac{1}{d'}(p_2')^*(\omega _X^{p-j})]\\
&=&\frac{1}{b_n}\pi ^*[G_n^*(\omega _Y^j)\wedge \omega _Y^{l-j}]\wedge \Gamma ^*(\omega _X^{p-j})\wedge \omega _X^{k-l-p+j}.
\end{eqnarray*}
In the above, we used that if $\tau :Y'\rightarrow Y$ is a generically finite surjective morphism and $\alpha$ is a finite union of points in $Y'$ then $\tau ^*\tau _*(\alpha )$ and $\deg (\tau )\alpha $ have the same class in $N^*(Y')$.  This is then applied to $Y'=\widetilde{G_n}$, $\tau =p_1$, and $\alpha =p_1^*(\omega _Y^{l-j})\wedge p_2^*(\omega _Y^j)$ (recall here that $\dim (Y)=l$). 

Taking  the sum all over $\Gamma$'s, and by observing that $(g^n)^*=a_nG_n^*$ and $\lim _{n\rightarrow\infty}(a_nb_n)^{1/n}= \lambda _0(g)$, we obtain the inequality 
\begin{eqnarray*}
\lambda _0(g)\lambda _p(f) \geq \max _{0\leq j\leq l,~0\leq p-j\leq k-l}\lambda _j(g)\lambda _{p-j}(f|\pi ),
\end{eqnarray*}
for all $p=0,\ldots ,\dim (X)$. 

To finish the proof of the theorem, we need to show that if $c>0$ is a constant such that 
\begin{eqnarray*}
c\lambda _p(f) \geq \max _{0\leq j\leq l,~0\leq p-j\leq k-l}\lambda _j(g)\lambda _{p-j}(f|\pi ),
\end{eqnarray*}
for all $p=0,\ldots ,\dim (X)$, then $c\geq \lambda _0(g)$. In fact, if this is the case, then (by choosing $p=0$) we must have $c\lambda _0(f)\geq \lambda _0(g)\lambda _0(f|\pi )$. From the proof of Theorem \ref{TheoremPrimitiveCorrespondence}, we have that $\lambda _0(f)=\lambda _0(f|\pi )$, and hence $c\geq \lambda _0(g)$ as wanted. 

\subsection{Proof of Theorem \ref{TheoremFurtherProperty}}
Let $\tau _2:Z\rightarrow X_2$ be an alteration (so $Z$ is smooth and $\tau _2$ is a generically finite regular morphism) so that  the composition maps $\tau _1=\varphi \circ \tau _2:Z\rightarrow X_1$, $\pi _2'=\pi _2\circ \tau _2:Z\rightarrow Y_2$ and $\pi _1'=\pi _1\circ  \tau _1:Z\rightarrow Y_1$ are all regular. 

1) Let us define two correspondences over $Z$: $h_1=\tau _1^*(f_1)$ and $h_2=\tau _2^*(f_2)$. We have two semi-conjugacies $(Z,h_1)\rightarrow (X_1,\deg (\tau _1)f_1)$ and $(Z,h_2)\rightarrow (X_2,\deg (\tau _2)f_2)$. By the last part of the proof of Lemma \ref{LemmaBirationalInvariant}, we have that $\lambda _p(h_1|\pi _1')=\lambda _p(h_1|\pi _2')$ for all $p$. Hence, by definition and the fact that $\deg (\tau _1)=\deg (\tau _2)\deg (\varphi )$ since $\tau _1=\varphi \tau _2$, we obtain
\begin{eqnarray*}
\lambda _p(f_1|\pi _1)&=&\frac{1}{\deg (\tau _1)}\lambda _p(h_1|\pi _1')=\frac{1}{\deg (\tau _2)\deg (\varphi )}\lambda _p(h_1|\pi _2'),\\
\lambda _p(f_2|\pi _2)&=&\frac{1}{\deg (\tau _2)}\lambda _p(h_2|\pi _2').
\end{eqnarray*}

First we show that $\Gamma _{h_1}-\Gamma _{h_2}$ is an effective variety, that is all irreducible components of $\Gamma _{h_2}$ are contained in $\Gamma _{h_1}$ and the multiplicity of that component in $\Gamma _{h_2}$ is bounded from above by the corresponding in $\Gamma _{h_1}$. To this end, by defining $f_2'=\varphi _2^*(f_1):X_2\rightarrow X_2$ we observe that $h_1=\tau _2^*\varphi ^*(f_1)=\tau _2^*(f_2')$, therefore to show that $\Gamma _{h_1}-\Gamma _{h_2}$ is an effective variety it suffices to show that $\Gamma _{f_2'}-\Gamma _{f_2}$ is effective. This follows easily from the fact that for a generic $x_2\in X_2$:
\begin{eqnarray*}
\varphi f_2(x_2)&=&f_1\varphi (x_2),\\
f_2'(x_2)&=&\varphi ^{-1}f_1\varphi (x_2).
\end{eqnarray*}

From this, it readily follows that $\Gamma _{h_1^n}-\Gamma _{h_2^n}$ is an effective variety for all $n\in \mathbb{N}$. Therefore, we obtain by definition of relative dynamical degrees that $\lambda _p(h_1|\pi _2')\geq \lambda _p(h_2|\pi _2')$ for all $p$. The above consideration then implies that $\deg (\varphi )\lambda _p(f_1|\pi _1)\geq \lambda _p(f_2|\pi _2)$  for all $p$. 

Now we finish the proof of 1). Applying the above inequality to $f_1^n$ and $f_2^n$ instead of $f_1$ and $f_2$, we obtain $\deg (\varphi )\lambda _p(f_1^n|\pi _1)\geq \lambda _p(f_2^n|\pi _2)$ for all $n\in \mathbb{N}$. From this and the fact (following from part 1) of Theorem \ref{TheoremRelativeDynamicalDegrees}) that $\lambda _p(f_1^n|\pi _1)=\lambda _p(f_1|\pi _1)^n$ and $\lambda _p(f_2^n|\pi _2)=\lambda _p(f_2|\pi _2)^n$ for all $n\in \mathbb{N}$, we have $\lambda _p(f_1|\pi _1)\geq \lambda _p(f_2|\pi _2)$ for all $p$.  

 2)  Assume that condition a) is satisfied, that is $X_1$ is smooth and $\pi _1$ is regular. We consider the correspondence $h=\tau _2^*(f_2)$. By Lemma \ref{LemmaSemiconjugacy1}, we have a semi-conjugacy $(Z,h=\tau _2^*(f_2))\rightarrow (X_2,\deg (\tau _2)f_2)$, which together with the given semi-conjugacy $(X_2,f_2)\rightarrow (X_1,f_1)$ gives also a semi-conjugacy $(Z,h)\rightarrow (X_1,\deg (\tau _2)f_1)$. By definition, we have $\lambda _p(h|\pi _2')=\deg (\tau _2)\lambda _p(f_2|\pi _2)$. Under condition a), Lemma \ref{LemmaSemiconjugacyInvariant} can be applied, and we obtain $\lambda _p(h|\pi _1')=\deg (\tau _2)\lambda _p(f_1|\pi _1)$. The argument at the end of the proof of Lemma \ref{LemmaBirationalInvariant} helps to show that $\lambda _p(h|\pi _1')=\lambda _p(h|\pi _2')$, and the latter gives that $\lambda _p(f_1|\pi _1)=\lambda _p(f_2|\pi _2)$ as wanted.

Assume that condition b) is satisfied. Then we can find a resolution $\tau :X_1'\rightarrow X_1$ such that $X_1'$ is smooth and the induced map $\pi _1\circ \tau :X_1'\rightarrow Y_1$ is regular. Since $\tau $ is birational, the semi-conjugacy $(X_2,f_2)\rightarrow (X_1,f_1)$ lifts to a semi-conjugacy $(X_2,f_2)\rightarrow (X_1',f_1'=\tau ^*(f_1))$. By Lemma \ref{LemmaBirationalInvariant}, $\lambda _p(f_1'|\pi _1\circ \tau )=\lambda _p(f_1|\pi )$. Hence this case is reduced to the previous case, and the proof is completed.

\subsection{Proof of Theorem \ref{TheoremTriangleInequality}}

By using pullback to an alteration $X_1$ of $X$ if necessary, we may assume that $X$ is smooth.  Since $f_1,f_2:X\rightarrow X$ are commutative, we have 
\begin{eqnarray*}
(f_1+f_2)^n=\sum _{i=0}^nC(n,i)f_1^if_2^{n-i},
\end{eqnarray*}
where $C(n,i)$ are the binomial coefficients. In the proof of Theorem \ref{TheoremRDD4}, by using $f_1^i$, $f_2^{n-i}$ and their composition in the place of $f^n$, $f^m$ and $f^{n+m}$, we find that there is a constant $A>0$ depending only on $X$ and $Y$ and the map $\pi$ such that in the notation of Subsection 5.1
\begin{eqnarray*}
\deg _p((f_1+f_2)^n)\leq A\sum _{i=0}^nC(n,i)\deg _p(f_1^i)\deg _p(f_2^{n-i}). 
\end{eqnarray*}
This, together with part 1) of Theorem \ref{TheoremRelativeDynamicalDegrees}, concludes the proof of Theorem \ref{TheoremTriangleInequality}. 

\section{Extensions and applications}

In this section we collect some extensions and applications. 

\subsection{Correspondences over non-projective varieties}

Let $K$ be an algebraic closed field, $X$ an (not necessarily projective) irreducible variety and $f:X\rightarrow X$ a correspondence. By de Jong's theorem, there is a generically finite rational map $\tau :X'\rightarrow X$ for which $X'$ is a smooth irreducible projective variety. We then define, as in the proof of Theorem \ref{TheoremRelativeDynamicalDegrees} 
\begin{eqnarray*}
\lambda _p(f)=\frac{1}{\deg (\tau )}\lambda _p(\tau ^*(f)).
\end{eqnarray*}
The proof of Theorem \ref{TheoremRelativeDynamicalDegrees} shows that this is well-defined. 

If $X$ is not projective but $Y$ is projective, we can define relative dynamical degrees similarly. 

\subsection{Correspondences over reducible varieties}

Let $K$ be an algebraic closed field, and $X,Y$  (reducible) projective varieties. Let $f:X\rightarrow X$ and $g:Y\rightarrow Y$ be dominant correspondences. Assume that $\pi :X\rightarrow Y$ is a dominant rational map such that $\pi \circ f^n=g^n\circ \pi $ for all $n\in \mathbb{N}$. (Recall from Section 3 that both definitions of iterates $f^n$ and semi-conjugacies for correspondences over reducible varieties are quite delicate.) Then relative dynamical degrees can be defined and satisfy all the properties in Theorem \ref{TheoremRelativeDynamicalDegrees}. The reason for this is that in proving these results, the property we need is that we have an equality between strict transforms $(f^n)_0\circ (f^m)_0=(f^{n+m})_0$ outside certain bad sets, and this is certainly true for the general case considered here even though the iterates $f^n$ are more involved to define. 

Let us elaborate on how to proceed. We discuss the case of dynamical degrees only, the general case of relative dynamical degrees is similar. Let $X=X_1\cup \ldots X_m$ be  a reducible projective variety, where $X_i$'s are distinct irreducible components of it. Let $f:X\rightarrow X$ be a dominant rational morphism. We claim that there are alterations $\tau _i:Z_i\rightarrow X_i$ for $i=1,\ldots ,m$, where $Z_i$ are distinct irreducible smooth varieties. In fact, we choose first for each  $i$ an alteration $Z_i'\rightarrow X_i$. If the $Z_i'$ are pairwise distinct, we are done. Otherwise, if $Z_i'=Z_j'$, we replace any one of these by the blowup at several points while keeping the other one unchanged. This procedure is done until the $Z_i'$ are pairwise distinct.  
 
For each component $f_{i,j}:X_i\rightarrow X_j$ of $f$, we consider the pullback $g_{i,j}:Z_i\rightarrow Z_j$, where the graph of $g_{i,j}$ is the pullback of $\Gamma _{f_{i,j}}$ by the generically finite morphism $\tau _i\times \tau _j:Z_i\times Z_j\rightarrow X_i\times X_j$. Define $d_{i,j}=$ the degree of the map $\tau _i\times \tau _j$ $=$ $d_id_j$, where $d_i$ is the degree of the morphism $\tau _i$. By Lemma \ref{LemmaSemiconjugacy1}, we have a semi-conjugacy of pairs $\tau _{i,j}=(\tau _i,\tau _j):(Z_i,Z_j,g_{i,j})\rightarrow (X_i,X_j,d_{i,j}f_{i,j})$. We define $d=\prod _{(i,j)\in I}d_{i,j}$, and 
\begin{eqnarray*}
g=\sum _{(i,j)}\frac{d}{d_{i,j}}g_{i,j}. 
\end{eqnarray*}
Let $Z=Z_1\cup \ldots \cup Z_m$, then $g$ is a correspondence on $Z$ and it is dominant since $f$ is dominant. Since $Z_i$ are distinct and $X_i$ are distinct, we see by definition that $f_{i,j}$ and $f_{k,l}$ can be composed if and only if $g_{i,j}$ and $g_{k,l}$ can be composed.  From this, it can be checked that we have a semi-conjugacy $\tau :(Z,g)\rightarrow (X,df)$ which moreover satisfies $\tau \circ g^n=f^n\circ \tau $, where $\tau =(\tau _1,\ldots ,\tau _m):(Z_1,\ldots ,Z_m)\rightarrow (X_1,\ldots ,X_m)$. Then, provided that we can define $\lambda _p(g)$, we then can assign
\begin{eqnarray*}
\lambda _p(f):=\frac{1}{d}\lambda _p(g).
\end{eqnarray*}
Now we indicate how to define $\lambda _p(g)$ and to show that the above definition is independent of the choice of $Z$. Let $\omega _{i}$ be an ample divisor on $Z_i$. We then define
\begin{eqnarray*}
\deg _p(g)=\sum _{(i,j)\in I}f_{i,j}^*(\omega _j^p)\wedge \omega _{i}^{k-p}.
\end{eqnarray*}  
Then it can be proven as before that the following limit exists
\begin{eqnarray*}
\lim _{n\rightarrow\infty}(\deg _p(g^n))^{1/n},
\end{eqnarray*}
and we assign $\lambda _p(g)$ to be this limit. 

Now we show that the definition of $\lambda _p(f)$, as defined above, is independent of the choice of $Z$. Let be given another $Z'=Z_1'\cup \ldots Z_m'$ together with alterations $\tau _i':Z_i'\rightarrow X_i$ and pullbacks $g_{i,j}=(\tau _i',\tau _j')^*(f_{i,j})$. Let $d_{i,j}'$ be the degree of the morphism $\tau _i'\times \tau _j':Z_i'\times Z_j'\rightarrow X_i\times X_j$, and $d'=\prod _{(i,j)\in I}d_{i,j}'$. Then the correspondence $g':Z'\rightarrow Z'$ is
\begin{eqnarray*}
g'=\sum _{(i,j)\in I} \frac{d'}{d_{i,j}'}g_{i,j}'.
\end{eqnarray*}

We can find another $\widetilde{Z}=(\widetilde{Z_1},\ldots ,\widetilde{Z_m})$ together with alterations $\varphi _i:\widetilde{Z_i}\rightarrow Z_i$ and $\varphi ':\widetilde{Z_i}\rightarrow Z_i'$ such that $\tau _i\circ \varphi _i=\tau _i'\circ \varphi _i'$, and we denote by $\widetilde{\tau _i}:\widetilde{Z_i}\rightarrow X_i$ this common morphism. We then define $\widetilde{g_{i,j}}:\widetilde{Z_i}\rightarrow \widetilde{Z_j}$ to be the pullback of $f_{i,j}$ by $(\widetilde{\tau _i},\widetilde{\tau _j}):(\widetilde{Z_i},\widetilde{Z_j}\rightarrow (X_i,X_j)$, $\widetilde{d_{i,j}}=$ the degree of $\widetilde{\tau _i}\times \widetilde{\tau _j}:\widetilde{Z_i}\times \widetilde{Z_j}\rightarrow X_i\times X_j$, and $\widetilde{d}=\prod _{(i,j)\in I}\widetilde{d_{i,j}}$. The correspondence $\widetilde{g}:\widetilde{Z}\rightarrow \widetilde{Z}$ is finally given by
\begin{eqnarray*}
\widetilde{g}=\sum _{(i,j)\in I}\frac{\widetilde{d}}{\widetilde{d_{i,j}}}\widetilde{g_{i,j}}.
\end{eqnarray*} 
Now we consider each $(i,j)\in I$ separately. We have a semi-conjugacy $\widetilde{g_{i,j}}\rightarrow \deg (\varphi _i,\varphi _j)g_{i,j}$, therefore we have a semi-conjugacy
\begin{eqnarray*}
\frac{\widetilde{d}}{\widetilde{d_{i,j}}}\widetilde{g_{i,j}}\rightarrow \frac{\widetilde{d}}{\widetilde{d_{i,j}}}\deg (\varphi _i,\varphi _j)g_{i,j}=\frac{\widetilde{d}}{{d_{i,j}}}g_{i,j}=d'\frac{{d}}{{d_{i,j}}}g_{i,j}.
\end{eqnarray*} 
Taking the sum, we have a semi-conjugacy $\widetilde{g}\rightarrow d'g$, and moreover we have as argued before semi-conjugacies $\widetilde{g}^n\rightarrow (d')^ng^n$ for all $n\in \mathbb{N}$. Similarly, we have semi-conjugacies $\widetilde{g}^n\rightarrow d^n(g')^n$. From this, all the results proven previously for correspondences over irreducible varieties apply.  

\subsection{Fields that are not algebraically closed}

Let $K$ be a field, not necessarily closed, and $\overline{K}$ its algebraic closure. Let $X$ be a variety over $K$, together with a dominant correspondence $f:X\rightarrow X$. Let $\overline{X}$ be the lift of $X$ and $\overline{f}$ the lift of $f$ to $\overline{K}$ . We note that even if $X$ is irreducible, $\overline{X}$ may not be so. Hence in general we need to work with correspondences over reducible varieties. We then can define dynamical degrees for $f$ in terms of that for $\overline{f}$. 

\subsection{The simplicity of the first dynamical degree}
Let $K$ be an algebraically closed field of arbitrary characteristic, and $X$ a projective manifold. Assume that resolution of singularities is available for varieties over $K$ up to $\dim (X)$. We obtain analogs of Theorems 1.1 and 1.2 in \cite{truong2}, whose proofs, which are similar to those in \cite{truong2}, are omitted. 

We say that a dominant rational map $f:X\rightarrow X$ is $1$-stable if for any $n\in \mathbb{N}$ the two pullback operators $(f^n)^*:N^1_{\mathbb{R}}(X)\rightarrow N^1_{\mathbb{R}}(X)$ and $(f^*)^n:N^1_{\mathbb{R}}(X)\rightarrow N^1_{\mathbb{R}}(X)$ are the same. On the one hand, not all dominant rational map are $1$-stable, a simple example is that of the standard Cremona map $\sigma :\mathbb{P}^2\rightarrow \mathbb{P}^2$ given by 
\begin{eqnarray*}
\sigma [x:y:z]=[\frac{1}{x}:\frac{1}{y}:\frac{1}{z}].
\end{eqnarray*}
On the other hand, there are many interesting dominant rational maps which  are $1$-stable, including regular selfmaps and pseudo-automorphisms in dimension $3$. The latter is the class of all birational maps $f:X\rightarrow X$, where $\dim (X)=3$, so that both $f$ and $f^{-1}$ have no exceptional divisors. 

\begin{theorem} Let $X\subset \mathbb{P}^N_K$ be a projective manifold of dimension $k$, and $f:X\rightarrow X$ a dominant rational map which is $1$-stable. Assume that $\lambda _1(f)^2>\lambda _2(f)$. Then $\lambda _1(f)$ is a simple eigenvalue of $f^*:N^1_{\mathbb{R}}(X)\rightarrow N^1_{\mathbb{R}}(X)$. Further, $\lambda _1(f)$ is the only eigenvalue of modulus greater than $\sqrt{\lambda _2(f)}$.  
\label{TheoremFirstDynamicalDegreeIsSimpleAlgebraicCase}\end{theorem}

We denote by $r_1(f)$ and $r_2(f)$ the spectral radii of the linear maps $f^*:N^1_{\mathbb{R}}(X)\rightarrow N^1_{\mathbb{R}}(X)$ and $f^*:N^2_{\mathbb{R}}(X)\rightarrow N^2_{\mathbb{R}}(X)$.
\begin{theorem}
Let $X\subset \mathbb{P}^N_K$ be a projective manifold, and let $f:X\rightarrow X$ be a dominant rational map. Assume that $f^*:N^2_{\mathbb{R}}(X)\rightarrow N^2_{\mathbb{R}}(X)$ preserves the cone of effective classes. Then  

1) We have $r_1(f)^2\geq r_2(f)$.

2) Assume moreover that $r_1(f)^2>r_2(f)$. Then $r_1(f)$ is a simple eigenvalue of $f^*:N^1_{\mathbb{R}}(X)\rightarrow N^1_{\mathbb{R}}(X)$. Further, $r_1(f)$ is the only eigenvalue of modulus greater than $\sqrt{r_2(f)}$.  

\label{TheoremTheCaseNotStableAlgebraicCase}\end{theorem}

\subsection{Surfaces and threefolds over a field of positive characteristic}

All results in this paper are valid for surfaces and threefolds over a field of positive characteristic, since resolutions of singularities for surfaces (see \cite{zariski1, abhyankar1}) and threefolds (see \cite{zariski2, abhyankar2, cutkosky, cossart-piltant1, cossart-piltant2}) over a field of positive characteristic are available. 

\subsection{Primitive correspondences}

We say that a dominant correspondence $f:X\rightarrow X$ is primitive if there is no semi-conjugacy $\pi :(X,f)\rightarrow (Y,g)$, where $Y$ is irreducible smooth and $0<\dim (Y)<\dim (X)$. Note that any dominant correspondence $f:X\rightarrow X$ always semi-conjugate to a multiple of the identity map of a point, and hence in the above definition we may consider the special case where $g$ is a multiple of a dominant rational map on $Y$, or more generally a correspondence $g$ so that $g^n$ is irreducible for infinitely many $n\in \mathbb{N}$. By the examples in Remark \ref{Remark1}, there are many semi-conjugacies $\pi :(X,f)\rightarrow (Y,g)$ where $g$ is a multiple of a rational map but $f$ is not a multiple of a rational map. In fact, such examples can always be constructed provided we are given two dominant rational maps $\pi :X\rightarrow Y$ and $g':Y\rightarrow Y$, where $\pi$ is not a birational map.   

We say that a dominant correspondence $g:Y\rightarrow Y$ is morphism-like, if $\lambda _0(g)\leq \lambda _1(g)$ and $g^n$ is irreducible for infinitely many $n\in \mathbb{N}$. For example, if $g$ is a multiple of a dominant rational map then it is morphism-like.  We say that a dominant correspondence $f:X\rightarrow X$ is weakly primitive if it cannot be semi-conjugate to a correspondence $g:Y\rightarrow Y$, where $0<\dim (Y)<\dim (X)$ and $g$ is a morphism-like. 

We have the following criterion for weakly primitivity. It says that if $f$ is not morphism-like, then it must be weakly primitive. 
\begin{theorem}
 If $\lambda _0(f)>\lambda _1(f)$, then $f$ is weakly primitive. 
\label{TheoremPrimitiveCorrespondence}\end{theorem}
\begin{proof}
This follows from Theorem \ref{TheoremWeakerProductFormula}, by using that  for a morphism-like $g$ we have $\lambda _1(g)\geq \lambda _0(g)$ and that $\lambda _0(f)=\lambda _0(f|\pi )$. The latter can be seen as follows. We have by definition
\begin{eqnarray*}
\lambda _0(f|\pi )=\lim _{n\rightarrow\infty}((f^n)^*(X).\pi ^*(\omega _Y^l).\omega _X^{k-l})^{1/n}=\lim _{n\rightarrow\infty}||(f^n)^*(X)||^{1/n}=\lambda _0(f).
\end{eqnarray*}
\end{proof}
We note that under the assumptions of this theorem, it may still happen that $f$ is imprimitive, being semi-conjugate to a correspondence which is not morphism-like. See Example 2 below for more detail. 

Similarly, when $f$ is a dominant rational map such that $\lambda _1(f)>\lambda _2(f)$, it can be shown as in \cite{oguiso-truong} that $f$ is (strongly) primitive. We also have the following consequence of the product formula and resolution of singularities for surfaces and threefolds. It can be applied to the case $\dim (X)\leq 4$. 
\begin{corollary}
Let $X,Y$ be irreducible projective varieties (not necessarily smooth), and $\pi :(X,f)\rightarrow (Y,g)$ a semi-conjugacy of dominant correspondences, where $g$ is a multiple of a rational map. Assume that $\dim (Y)\leq 3$. Then the product formula is satisfied. 
\label{CorollaryProductFormula}\end{corollary}

{\bf Example 1.} For the Hurtwitz correspondences $f$, it was shown in \cite{ramadas} that $\lambda _0(f)$ is the largest dynamical degree, hence Theorem \ref{TheoremPrimitiveCorrespondence} may be applicable. 

{\bf Example 2.} Let $f':X\rightarrow X$ be a dominant rational map such that $\lambda _k(f')>\lambda _{k-1}(f')$, where $k=\dim (X)$ . (For example, $f'$ is a regular morphism of $\mathbb{P}^k$ of degree $d\geq 2$.) Let $f=(f')^{-1}$ be the reverse correspondence of $f'$, i.e. the graph of $f$ is the image of the graph of $f'$ by the automorphism $(x_1,x_2)\rightarrow (x_2,x_1)$ of $X\times X$. Then $\lambda _0(f)=\lambda _{k}(f')>\lambda _{k-1}(f')=\lambda _1(f)$. Therefore, Theorem \ref{TheoremPrimitiveCorrespondence} is applicable.   

From a given $f:X\rightarrow X$ which satisfies the assumptions of Theorem \ref{TheoremPrimitiveCorrespondence}, we can construct more in the following manner. Let $f_1,\ldots ,f_m:X\rightarrow X$ be dominant correspondences for which any pair among $f,f_1,\ldots ,f_m$ commute. (For example, choose the $f_i$'s among the iterates of $f$.) Then the new correspondence $af+a_1f_1+\ldots +a_mf_m$ with $a>>\max \{a_1,\ldots ,a_m\}$ will do the job, thanks to Theorem \ref{TheoremTriangleInequality}. 

In these examples, it is possible that $f$ is still imprimitive. In fact, let $k\geq 2$, $X=\mathbb{P}^k$, $Y=\mathbb{P}^{k-1}$, and $\pi :X\rightarrow Y$ the dominant rational map $[x_0:x_1:\ldots :x_{k-1}:x_k]\mapsto [x_0:\ldots :x_{k-1}]$. Let $d\geq 2$, $f':X\rightarrow X$ be the surjective morphism $[x_0:\ldots :x_k]\mapsto [x_0^d:\ldots :x_{k}^d]$, and $g':Y\rightarrow Y$ the surjective morphism $[x_0:\ldots :x_{k-1}]\mapsto [x_0^d:\ldots :x_{k-1}^d]$. Let $f$ be the reverse of $f'$ and $g$ the reverse of $g'$. Then it can be checked that we have a  semi-conjugacy $\pi :(X,f)\rightarrow (Y,g)$. Note that the correspondence $g$ is not morphism-like: while $g^n$ (being the reverse of $(g')^n$) is irreducible for all $n$, the condition $\lambda _0(g)\leq \lambda _1(g)$ is not satisfied.

{\bf Example 3.} Even in the case where the dynamical degrees of a correspondence do not satisfy the assumptions of Theorem \ref{TheoremPrimitiveCorrespondence}, a full use of the product formula for dynamical degrees may still help to show that a certain correspondence is primitive. For example, let $f:X\rightarrow X$ be a correspondence on a surface.  Assume that there is an irreducible curve $Y$ together with  a dominant rational map $g':Y\rightarrow Y$ and a dominant rational map $\pi :X\rightarrow Y$ so that $\pi \circ f=g\circ \pi $, where $g=ag'$. By Corollary \ref{CorollaryProductFormula}, the product formula for relative dynamical degrees is valid in this situation, even when $Y$ is not smooth. Therefore, we obtain 
\begin{eqnarray*}
a\lambda _0(f)&=&\lambda _0(g)\lambda _0(f|\pi ),\\
a\lambda _1(f)&=&\max \{\lambda _0(g)\lambda _1(f|\pi ),\lambda _1(g)\lambda _0(f|\pi )\},\\
a\lambda _2(f)&=&\lambda _1(g)\lambda _1(f|\pi ).
\end{eqnarray*}
In particular
\begin{eqnarray*}
a^2\lambda _1(f)^2\geq (\lambda _0(g)\lambda _1(f|\pi ))(\lambda _1(g)\lambda _0(f|\pi ))=(\lambda _0(g)\lambda _0(f|\pi ))(\lambda _1(g)\lambda _1(f|\pi ))=a^2\lambda _0(f)\lambda _2(f).
\end{eqnarray*}
Hence, provided $\lambda _1(f)^2<\lambda _0(f)\lambda _2(f)$ (that is, when the log-concavity fails), $f$ cannot be semi-conjugate to a multiple of a dominant rational map over an irreducible curve $Y$. 

In Remark \ref{Remark1}, we constructed many examples of correspondences whose dynamical degrees violate the log-concavity. Here is one specific example. We let $h:\mathbb{P}^2\rightarrow \mathbb{P}^2$ be a regular morphism of algebraic degree $d\geq 2$. Let $a$  be a positive integer. Consider $f=h+a\Delta $, where $\Delta $ is the diagonal. Then $\lambda _0(f)=1+a$, $\lambda _1(f)=d+a$ and $\lambda _2(f)=d^2+a$. It can be checked that $\lambda _1(f)^2<\lambda _0(f)\lambda _2(f)$ for all choices of $a$.  

{\bf Example 4.} Similar to Example 3, if $\pi :(X,f)\rightarrow (Y,ag')$ is a semi-conjugacy, where $\dim (X)=3$, $0<\dim (Y)<3$, and $g':Y\rightarrow Y$ is a dominant rational map, then we must have $\lambda _0(f)\lambda _3(f)\leq \lambda _1(f)\lambda _2(f)$. A construction in the same line as that in Remark \ref{Remark1} and Example 2 provides many examples where the previous inequality is violated.

\end{document}